\documentclass[12pt]{amsart}

\newtheorem{theorem}{Theorem}[section]
\newtheorem{lemma}[theorem]{Lemma}
\newtheorem{corollary}[theorem]{Corollary}

\theoremstyle{plain}
\newtheorem{definition}[theorem]{Definition}
\newtheorem{example}[theorem]{Example}
\newtheorem{remark}[theorem]{Remark}
\newtheorem{hypothesis}[theorem]{Hypothesis}

\newtheorem{question}[theorem]{Question}

\let\olddefin\definition
\renewcommand{\definition}{\olddefin\normalfont}
\let\oldlemma\lemma
\renewcommand{\lemma}{\oldlemma\normalfont}
\let\oldtheorem\theorem
\renewcommand{\theorem}{\oldtheorem\normalfont}
\let\oldcoro\corollary
\renewcommand{\corollary}{\oldcoro\normalfont}
\let\oldquestion\question
\renewcommand{\question}{\oldquestion\normalfont}

\let\oldremark\remark
\renewcommand{\remark}{\oldremark\normalfont}
\let\oldexample\example
\renewcommand{\example}{\oldexample\normalfont}
\let\oldhypothesis\hypothesis
\renewcommand{\hypothesis}{\oldhypothesis\normalfont}

\theoremstyle{definition}

\theoremstyle{remark}

\numberwithin{equation}{section}

\usepackage{fancyhdr}
\usepackage{amscd}
\usepackage{amsmath}
\usepackage{amsthm}
\usepackage{amssymb}
\begin{document}

\title{K-theoretic Chow groups of derived categories of schemes \\ \vspace{3 mm} \small {On a question by Green-Griffiths}}

\author{Sen Yang}
\address{Yau Mathematical Sciences Center \\
Tsinghua University \\
Beijing, China\\
}
\email{syang@math.tsinghua.edu.cn; senyangmath@gmail.com}

\date{}

\maketitle

\begin{abstract}
Based on Balmer's tensor triangular Chow group [5], we propose K-theoretic Chow groups of derived categories of noetherian schemes and their Milnor variants for regular schemes and their thickenings. We discuss functoriality and show that our Chow groups agree with the classical ones [12] for regular schemes. We also define tangent spaces to our Chow groups as usually and identify them with cohomology groups of absolute differentials. 

Moreover, we extend Bloch-Quillen identification from regular schemes to their thickenings. This gives a positive answer to a question by Green-Griffiths, see question 1.1 below. We continue exploring the geometry of these K-theoretic Chow groups in forthcoming papers.
  
\end{abstract}

\tableofcontents

\section{Introduction}
\label{Introduction}

The study of Chow groups of algebraic cycles is a central topic in algebraic geometry. 
Relating with several important conjectures, including Hodge conjecture and Bloch-Beilinson conjecture,
Chow groups are very difficult to study. 

Recently, Green-Griffiths made progress on studying tangent spaces to Chow groups [15,16]. Fundamental to their work is the Soul\'e  variant of the Bloch-Quillen identification
\[
 CH^{p}(X)= H^{p}(X,\underline{K}^{M}_{p}(O_{X})) \ mod \ torsion,
\]
here $X$ is a smooth projective variety over $\mathbb{C}$, $\underline{K}^{M}_{p}(O_{X})$ is the 
Milnor K-theory sheaf associated to the presheaf
\[
  U \to K^{M}_{p}(O_{X}(U)).
\]

In [15], Green-Griffiths suggested that it would be interesting to extend Bloch-Quillen
 identification to infinitesimal thickening $(X,O_{X}[t]/(t^{m+1}))$. 
\begin{question} Green-Griffiths

Let $X_{m}$ denote the thickening $(X,O_{X}[t]/(t^{m+1}))$, do we have the following identification?
{\small
\[
 CH^{p}(X_{m})= H^{p}(X_{m},\underline{K}^{M}_{p}(O_{X}[t]/(t^{m+1}))) \ mod \ torsion.
\]
}
\end{question}

The answer is ``No'' for the calssical Chow groups, since the classical Chow groups can't detect nilpotent. To be precise, we have
\[
 CH^{p}(X)=CH^{p}(X_{m}).
\]

On the other hand, we know that
\[
 H^{p}(X,\underline{K}^{M}_{p}(O_{X}))\neq H^{p}(X_{m},\underline{K}^{M}_{p}(O_{X}[t]/(t^{m+1}))).
\]
In fact, 
{\small
\[
 H^{p}(X_{m},\underline{K}^{M}_{p}(O_{X}[t]/(t^{m+1}))) = H^{p}(X,\underline{K}^{M}_{p}(O_{X})) \oplus H^{p}(X,(\Omega_{X/ \mathbb{Q}}^{p-1})^{\oplus m}).
\]
}
here $\Omega_{X/ \mathbb{Q}}^{p-1}$ is the absolute differentials. 

This inspires us to propose a new definition of Chow groups capturing the nilpotent which is useful for studying deformation problems. Our starting point is to look at the derived category $D^{perf}(X)$ obtained from the exact category of perfect complexes of $O_{X}$-modules. It is obvious that the derived category $D^{perf}(X)$ is different from $D^{perf}(X_{m})$. 

In our approach, the derived category $D^{perf}(X)$ is considered as a tensor triangulated category, see [example 2.2]. Now, we turn to another side of this story, tensor triangular geometry. Mainly developed by P.Balmer and his collaborators, tensor triangular geometry is the study of tensor triangulated categories by algebraic geometry methods. Although a relatively new subject in its early stage, tensor triangular geometry has shown its power in studying algebraic geometry, modular representation theory and etc. For a good survey of this wonderful subject and its achievement, we refer the readers to Balmer's ICM talk [6].

It is Balmer's beautiful reconstruction theorem [1,Theorem 6.3] that opens the new land of this exciting theory. As recalled below, Balmer's reconstruction theorem says that a scheme can be reconstructed from its associated tensor triangulated category $D^{perf}(X)$.

\begin{theorem}

[Balmer]

Let $X$ be a quasi-compact and quasi-seperated scheme. We have an isomorphism $Spec(D^{perf}(X)) \simeq X$ of ringed spaces.
\end{theorem}  

Since one can reconstruct the scheme $X$ from the tensor triangulated category $D^{perf}(X)$, then it is reasonable to define Chow groups in terms of $D^{perf}(X)$, considered as a tensor triangulated category. One would like to have a functor $CH_{q}(D^{perf}(-))$ for schemes with good functorial properties, flat pull-back , proper push-forward and etc.
Moreover, one should have $CH_{q}(D^{perf}(X))= CH^{q}(X)$ for ``nice" schemes.

Such a construction has been proposed in [5] by P.Balmer and followed by S.Klein [19]. Balmer's new insight is to allow the coefficients of  {\it q-cycles} to lie in Grothendieck groups of suitable triangulated categories. To be precise, one filters tensor triangulated category $D^{perf}(X)$ by (-co)dimension of support. Then {\it q-cycles} is defined to be Grothendieck groups of idempotent completion of the {\it q-th} Verdier quotient of the filtration. See [definition 2.13, 2.14] for precise definitions.  

As pointed out in [4], taking idempotent completion can result in the appearance of negative K-groups. 
In order to include this important information into our study, we propose our K-theoretic Chow groups of $D^{perf}(X)$ by slightly modifying Balmer's, see [definition 3.6]. Our Chow groups indeed are subgroups of Balmer's.
Moreover, our Chow groups are cohomology groups of Gersten complexes. This is also guided by Quillen's proof of Bloch's formula [25].

Our main results are as follows.
\begin{itemize}

 \item Definitions. We propose definitions of K-theoretic Chow groups of derived categories of noetherian schemes, see definition 3.6. And we also define Milnor K-theoretic Chow groups of derived categories of regular schemes and thickenings, see definition 4.4 and 4.30. Milnor Chow groups of 0-cycles are discussed in appendix.
 \item Functoriality. Flat pull-back and proper push-forward are discussed in section 3.3.
 \item Agreement. We show that our (Milnor)Chow groups of derived categories agree with the classical ones for regular schemes, see theorem 3.8, theorem 4.31 and theorem 5.2.
 \item Bloch's formula. We show our (Milnor)Chow groups of derived categories satisfy Bloch's formula for regular schemes and their thickenings, see theorem 4.25 and theorem 4.32. This provides a positive answer to the above question by Green-Griffiths.
 \item We define tangent spaces to our Chow groups as usually, while the classical Chow groups can't do. We also identify tangent spaces to our Chow groups with cohomology groups of absolute differentials. See definition 4.26, theorem 4.27, definition 4.33 and theorem 4.34. 
\end{itemize}

The search of generalized {\it cycles} and {\it Chow groups} is also motivated by intersection theory. 
The classical construction of the Chow ring fails when one deals with singular algebraic varieties. Among others, Levine-Weibel and Pedrini-Weibel defined relative Chow groups [20,23]. They also proved their relative Chow groups satisfied Bloch's formula [21,24]. Comparing our Chow groups of derived categories with relative Chow groups, we suggest a question (on page 15) which might be closely related with Gersten conjecture.

The idea of using derived category to define {\it algebraic cycles} was also suggested by Thomason [31].  To honor Thomason whose higher K-theory of derived categories of schemes has been widely accepted, we cite his idea here: `` I seek to define a good intersection ring of
‘algebraic cycles’ on schemes X where the classical construction of the Chow
ring fails, for example on singular algebraic varieties or on regular schemes flat
and of finite type over Z. Inspired by the superiority of Cartier divisors over
Weil divisors and by recent progress in local intersection theory, I believe the good notion of ‘algebraic n-cycle’ is that of those perfect complexes in some triangulated subcategory $A^{n} \subseteq D(X)_{parf}$ which remains to be defined."

We shall not discuss intersection theory in this paper. For a good survey of Chow groups and intersection theory, we refer to Gillet [14].

\textbf{Acknowledgements}
I sincerely thank professor P.Balmer for precious discussions and correspondence. His work on Chow groups of tensor triangulated categories [5] and previous work on tensor triangular geometry [1,2,3,4] make our paper become possible. I also sincerely thank professor M.Schlichting for teaching me K-theory and for inviting me to visit Warwick Mathematics Institute where some work have been done. Thanks Warwick Mathematics Institute for its hospitality. Thanks professor B.Totaro, professor C.Pedrini and S.Klein for helpful comments.

I am very grateful to professor G.Green and professor P.Griffiths for enlightening discussions and for telling me their paper [15] which started this paper.

Many thanks to J.W.Hoffman and B.Dribus for collaboration whose contribution go beyond [11]. Many thanks to department of mathematics of Louisiana State University for financial support. The computational results of 4.2 and 4.3 have been done in my thesis [36].

\textbf{Notations and conventions}.
$X$ is a $d$-dimensional noetherian scheme of finite type over a field $k$, if not stated otherwise. $Speck[\varepsilon]$ denotes the dual number, $\varepsilon^{2}=0$.

\section{Tensor triangular geometry}
\label{Tensor triangular geometry}
We begin with recalling basic definitions and examples of tensor triangular geometry in section 2.1. Following [5], we recall Balmer's K-theoretic Chow groups of tensor triangulated categories in section 2.2.
\subsection{Background}
\label{Background}

\begin{definition}[6]
A tensor triangulated category $(\mathcal{L}, \otimes, \amalg)$ is a triangulated category
 $\mathcal{L}$ equipped with a monoidal structure: $\mathcal{L} \otimes \mathcal{L} \to \mathcal{L}$ with 
unit object $\amalg$. We assume that $-\otimes-$ exact in each variable, i.e. both functors $ a \otimes -:\mathcal{L} \to \mathcal{L}$ and $- \otimes a:\mathcal{L} \to \mathcal{L}$ are exact for every $a \in \mathcal{L}$. Let $\sum$ denote the {\it suspension} of $\mathcal{L}$, we assume that natural isomorphisms $(\sum a)\otimes b \cong \sum (a \otimes b)$ and $ a \otimes (\sum b) \cong \sum (a \otimes b)$ compatible in that the two ways from $(\sum a)\otimes (\sum b)$ to $\sum^{2}(a \otimes b)$ only differ by a sign.

Although some of the theory holds without further assumption, we are going
to assume moreover that  is symmetric monoidal : $ a \otimes b = b \otimes a $.

\end{definition}

Examples of tensor triangulated categories can be found from algebraic geometry, motivic theory, modular representation theory etc. For our main interest, we recall the following standard example from algebraic geometry. More examples have been discussed in Balmer's ICM talk[6].

\begin{example}[6]
Let $X$ be a scheme, here always assumed quasi-compact and quasi-separated
(i.e. $X$ admits a basis of quasi-compact open subsets). A complex of $O_{X}$-modules is called perfect if it is locally quasi-isomorphic to a bounded complex of finite generated
projective modules. Then $\mathcal{L} = D^{perf}(X)$, the derived category of perfect
complexes over X, is a tensor triangulated category. See SGA6 [26] or Thomason [31]. The tensor $\otimes = \otimes ^{L}_{O_{X}}$ is the left derived tensor product and the unit $\amalg$ is $O_{X}$, considered as a complex concentrated  in degree $0$. 

When $X = Spec(A)$ is affine, $\mathcal{L} = D^{perf}(X) \cong K^{b}(A-proj)$, is the homotopy category
of bounded complexes of finite generated projective $A$-modules.
\end{example}

Another interesting example is Voevodsky's derived category of geometric motives.
\begin{example}[33, 6]
Let $S$ be the spectrum of a perfect field. Then
$\mathcal{L} = DM_{gm}(S)$, Voevodsky's derived category of geometric motives over $S$, is a
tensor triangulated category. 
\end{example}

The basic idea for studying tensor triangulated categories is to construct a topological space for every 
tensor triangulated category $\mathcal{L}$, called the tensor spectrum of $\mathcal{L}$, in which every object $b$ of $\mathcal{L}$ would have a support.

\begin{definition}[6]
A non-empty full subcategory $\mathcal{J} \subset \mathcal{L}$ is a triangulated subcategory
if for every distinguished triangle $ a \to b \to c \to \sum a$ in $\mathcal{L}$, when two out of $a, b, c$ belong to $\mathcal{J}$, so does the third.

$\mathcal{J}$ is called {\it thick} if it is stable by direct
summands : $a \oplus b \in \mathcal{J} \Rightarrow a, b \in \mathcal{J}$ and triangulated. 

$\mathcal{J}$ is {\it $\otimes$-ideal} if $\mathcal{L} \otimes \mathcal{J} \subset \mathcal{J}$; it is called {\it radical} if $a^{\otimes n} \in \mathcal{J} \Rightarrow a \in \mathcal{J}$.
\end{definition}

\begin{definition}[2]
A thick $\otimes$-ideal $\mathcal{P} \subset \mathcal{L}$ is called {\it prime} if it is proper($\amalg \notin \mathcal{P}$) and if $a \otimes b \in \mathcal{P}$ implies $a \in \mathcal{P}$, $b \in \mathcal{P}$.
\end{definition}

The spectrum of $\mathcal{L}$ is the set of primes:
\begin{align*}
 Spc(\mathcal{L}) = \{ \mathcal{P} \subsetneq \mathcal{L} \mid \mathcal{P} \ is \ a \ prime \}.
\end{align*}

The support of an object $a \in \mathcal{L}$ is defined as :
\[
 supp(a):= \{\mathcal{P} \in Spc(\mathcal{L}) \mid a \not\in  \mathcal{P} \}.
\]
The complement $U(a):= \{\mathcal{P} \in Spc(\mathcal{L}) \mid a \in \mathcal{P} \}$, for all $ a \in \mathcal{L}$, defines an open basis of the topology of $Spc(\mathcal{L})$.

We recall the following useful condition on $\mathcal{L}$ for later use.
\begin{definition}[6]
A tensor triangulated category $\mathcal{L}$ is {\it rigid} if there exists an exact functor $D : \mathcal{L}^{op} \to \mathcal{L}$ and a natural isomorphism $Hom_{\mathcal{L}}(a \otimes b, c) \cong Hom_{\mathcal{L}}(b, Da \otimes c)$ for every $ a, b, c \in \mathcal{L}$.
\end{definition}

\begin{hypothesis}
From now on, we assume our tensor triangulated category $\mathcal{L}$ to be
essentially small, rigid and idempotent complete.
\end{hypothesis}

\begin{definition}[5]
A rigid tensor triangulated category $\mathcal{L}$ is called
{\it local} if $a \otimes b = 0$ implies $a = 0$ or $b = 0$. 
\end{definition}

\begin{example}[5]
For every prime $\mathcal{P} \in Spc(\mathcal{L})$, the following tensor triangulated category
is local in the above sense : 
\[
 \mathcal{L}_{\mathcal{P}} := (\mathcal{L}/\mathcal{P})^{\#},
\]
where $\mathcal{L}/\mathcal{P}$ denote the Verdier quotient and $(-)^{\#}$ the idempotent completion.
\end{example}

\begin{definition}[5]
Assuming that $\mathcal{L}$ is local and that $Spc(\mathcal{L})$ is noetherian, the open
complement of the unique closed point $\{\ast \}$ in $Spc(\mathcal{L})$ is quasi-compact. This 
one-point subset corresponds to the minimal non-zero thick $\otimes$-ideal
\[
 Min(\mathcal{L}) := \{ a \in \mathcal{L} \mid supp(a) \subset \{\ast \} \}.
\]
\end{definition}
These are the objects with minimal possible support (empty or a point).

\subsection{Balmer's K-theoretic Chow group}
\label{Balmer's K-theoretic Chow group}
Now, we recall Balmer's K-theoretic Chow group of tensor triangulated categories. We begin with the definition of  {\it dimension function}.
\begin{definition}[5]

A {\it dimension function} on the space $Spc(\mathcal{L})$ is a map dim: $Spc(\mathcal{L}) \to \mathbb{Z}\cup \{ \pm\infty \}$ satisfying the following two conditions:
  \begin{itemize}

\item $\mathcal{P} \subseteq \mathcal{Q}$ implies $dim(\mathcal{P}) \leq  dim(\mathcal{Q})$.

\item $\mathcal{P} \subseteq \mathcal{Q}$ and  $dim(\mathcal{P}) =  dim(\mathcal{Q}) \in \mathbb{Z}$ 
     imply $\mathcal{P} = \mathcal{Q}$.
     
  \end{itemize}   
  
 Examples are Krull dimension of $\overline{\{ \mathcal{P} \}}$ in $Spc(\mathcal{L})$, or the opposite of its Krull codimension.
\end{definition}

Assuming dim(-) is clear from the context, we shall use the notation
\[
 Spc(\mathcal{L})_{(p)}:= \{\mathcal{P} \in Spc(\mathcal{L}) \mid dim(\mathcal{P}) = p \}.
\]

\begin{theorem}[3]

For a closed subset $Y \subset Spc(\mathcal{L})$, we set $dim(Y) = Sup \{dim(\mathcal{P}) \mid \mathcal{P} \in Y \}$ and consider the filtration $\dots \subset \mathcal{L}_{(p)} \subset \mathcal{L}_{(p+1)} \subset \dots \subset \mathcal{L}$ by dimension of support
\[
  \mathcal{L}_{(p)} := \{a \in \mathcal{L} \mid dim(supp(a)) \leq p \}.
\]
For every integer $p \in \mathbb{Z}$, we have the following equivalence induced by localization
\[
  (\mathcal{L}_{(p)}/\mathcal{L}_{(p-1)})^{\#} \simeq \bigsqcup_{\mathcal{P} \in Spc(\mathcal{L})_{(p)}}Min(\mathcal{L}_{\mathcal{P}})
\]
where $\mathcal{L}_{(p)}/\mathcal{L}_{(p-1)}$ is the Verdier quotient and $(-)^{\#}$ the idempotent completion.
\end{theorem}

With the above preparation, we are ready to recall Balmer's K-theoretic Chow group.

\begin{definition}[5]
Let $p \in \mathbb{Z}$, one define K-theoretic {\it p-cycles} associated to the tensor triangulated category $\mathcal{L}$ to be 
\[
 Z_{p}(\mathcal{L}) := K_{0}(\mathcal{L}_{(p)}/\mathcal{L}_{(p-1)})^{\#})= \bigoplus_{\mathcal{P} \in Spc(\mathcal{L})_{(p)}}K_{0}(Min(\mathcal{L}_{\mathcal{P}})),
\] 
where $K_{0}$ is the Grothendieck K-group(the quotient of the monoid of isomorphism class $[a]$ of objects under $\oplus$, by the submonoid of those $[a] + [\sum b] + [c]$ for which there exists a distinguish triangle $ a \to b \to c \to \sum a$). 
\end{definition}

According to [5], a K-theoretic {\it p-cycles} can be written as $\sum_{\mathcal{P}}\lambda_{\mathcal{P}}\cdot \overline{\{\mathcal{P}\}}$, for $\lambda_{\mathcal{P}} \in K_{0}(Min(\mathcal{L}_{\mathcal{P}}))$. Balmer's  new insight is to allow coefficients $\lambda_{\mathcal{P}}$ to live in the Grothendieck groups of $Min(\mathcal{L}_{\mathcal{P}})$.

\begin{definition}[5]
Let $p \in \mathbb{Z}$, we use  Ker$(i)$ denote the Kernel of $K_{0}(\mathcal{L}_{(p)}) \xrightarrow{i} K_{0}(\mathcal{L}_{(p+1)})$:

\[
  Ker(i) \to K_{0}(\mathcal{L}_{(p)}) \xrightarrow{i} K_{0}(\mathcal{L}_{(p+1)}).
\]

 The K-theoretic {\it p-boundaries} $B_{p}(\mathcal{L})$ is defined as the image of $Ker(i)$ in $Z_{p}(\mathcal{L})$
 \[
   B_{p}(\mathcal{L}):= j \circ Ker(i),
 \]
where $j: K_{0}(\mathcal{L}_{(p)}) \to K_{0}(\mathcal{L}_{(p)}/\mathcal{L}_{(p-1)})^{\#})(=Z_{p}(\mathcal{L}))$.

The K-theoretic {\it Chow group of p-cycles in $\mathcal{L}$}, denoted $CH_{p}(\mathcal{L})$, is defined 
to be the quotient of {\it p-cycles} by {\it p-boundaries}:

\[
  CH_{p}(\mathcal{L}) = \dfrac{Z_{p}(\mathcal{L})}{B_{p}(\mathcal{L})}.
\]

\end{definition}

\section{K-theoretic Chow groups of derived categories of schemes}
\label{K-theoretic Chow groups of derived categories of schemes}
In this section, we propose and study K-theoretic Chow groups of derived categories of noetherian schemes. In section 3.1, we propose K-theoretic Chow groups by slightly modifying Balmer's. We show that our K-theoretic Chow groups recover the classical ones for regular schemes in section 3.2. Functoriality is discussed in section 3.3. Comparing with the relative Chow groups, we suggest a question closely related with Gersten conjecture in 3.4.

\subsection{Definition}
\label{Definition}
Let $X$ be a noetherian scheme of finite Krull dimension $d$. As explained in Example 2.2, the derived category $\mathcal{L}=D^{perf}(X)$ is a tensor triangulated category.  With chosen {\it dimension function} on $D^{perf}(X)$, one can filter this category
$\dots \subset \mathcal{L}_{(p)} \subset \mathcal{L}_{(p+1)} \subset \dots \subset \mathcal{L}$  by dimension of support
\[
  \mathcal{L}_{(p)} := \{a \in \mathcal{L} \mid dim(supp(a)) \leq p \}.
\]

The Verdier quotient $\mathcal{L}_{(p)}/\mathcal{L}_{(p-1)}$ doesn't have good description when $X$ is singular. However, theorem 2.12 guides us to look at the idempotent completion  $(\mathcal{L}_{(p)}/\mathcal{L}_{(p-1)})^{\#}$. 

To fix some notations, for every $i \in \mathbb{Z}$, we define $X_{(i)}=\{x \in X \mid dim \overline{\{ x \}} = i \}$. We further assume the {\it dimension function} satisfy: $-d \leq dim(-) \leq d$.
\begin{theorem}
 [3]

For each $p \in \mathbb{Z}$, localization induces an equivalence
\[
 (\mathcal{L}_{(p)}/\mathcal{L}_{(p-1)})^{\#}  \simeq \bigsqcup_{x \in X_{(p)}}D_{{x}}^{perf}(X)
\]
between the idempotent completion of the quotient $\mathcal{L}_{(p)}/\mathcal{L}_{(p-1)}$ and the coproduct over $x \in X_{(p)}$ of the derived category of the $ O_{X,x}$-modules with homology supported on the closed point $x \in spec(O_{X,x})$.
\end{theorem}


The short sequence
\[
 \mathcal{L}_{(p-1)} \to \mathcal{L}_{(p)} \to (\mathcal{L}_{(p)}/\mathcal{L}_{(p-1)})^{\#},
\]
which is exact up to summand, induces a long exact sequence:
{\small
\[
 \dots \to K_{n}(\mathcal{L}_{(p-1)}) \xrightarrow{i}  K_{n}(\mathcal{L}_{(p)}) \xrightarrow{j} K_{n}((\mathcal{L}_{p}/\mathcal{L}_{p-1})^{\#}) \xrightarrow{k}  K_{n-1}(\mathcal{L}_{(p-1)}) \to \dots
\]
 }
 
We can apply Balmer's definition 2.13 and 2.14  to this setting.
\begin{definition}[5]
Let $p \in \mathbb{Z}$, one defines K-theoretic {\it p-cycles} associated to the tensor triangulated category $\mathcal{L} = D^{perf}(X)$ to be(with respect to chosen {\it dimension function})
\[
 Z_{p}(\mathcal{L}) := K_{0}(\mathcal{L}_{(p)}/\mathcal{L}_{(p-1)})^{\#})= \bigoplus_{x \in X_{(p)}}K_{0}(O_{X,x} \ on \ x).
\] 
\end{definition}

\begin{definition}[5]
Let $p \in \mathbb{Z}$, we use  $Ker(i)$ denote the Kernel of $K_{0}(\mathcal{L}_{(p)}) \xrightarrow{i} K_{0}(\mathcal{L}_{(p+1)})$. The K-theoretic {\it p-boundaries} $B_{p}(\mathcal{L})$ is defined as the image of $Ker(i)$ in $Z_{p}(\mathcal{L})$
 \[
   B_{p}(\mathcal{L}):= j \circ Ker(i),
 \]
where $j: K_{0}(\mathcal{L}_{(p)}) \to K_{0}(\mathcal{L}_{(p)}/\mathcal{L}_{(p-1)})^{\#})(=Z_{p}(\mathcal{L}))$.

The K-theoretic {\it Chow group of p-cycles in $\mathcal{L}$}, denoted $CH_{p}(\mathcal{L})$, is defined 
to be the quotient of {\it p-cycles} by {\it p-boundaries}

\[
  CH_{p}(\mathcal{L}) = \dfrac{Z_{p}(\mathcal{L})}{B_{p}(\mathcal{L})}.
\]

\end{definition}

As pointed in [4], the long exact sequence
{\small
\[
 \dots \to K_{n}(\mathcal{L}_{(p-1)}) \xrightarrow{i}  K_{n}(\mathcal{L}_{(p)}) \xrightarrow{j} K_{n}((\mathcal{L}_{p}/\mathcal{L}_{p-1})^{\#}) \xrightarrow{k}  K_{n-1}(\mathcal{L}_{(p-1)}) \to \dots
\]
 }
produces an exact couple as usually and then gives rise to the associated coniveau spectral sequence with $E_{1}$-term:
\[
 E_{1}^{p,q} = K_{-p-q}((\mathcal{L}_{(-p)}/\mathcal{L}_{(-p-1)})^{\#}).
\]
The differential $d$ is the composition $d = j \circ k$ as usual
{\footnotesize
\[
d_{1}^{p,q}: K_{-p-q}((\mathcal{L}_{(-p)}/\mathcal{L}_{(-p-1)})^{\#}) \xrightarrow{k} K_{-p-q-1}(\mathcal{L}_{(-p-1)}) \xrightarrow{j}  K_{-p-q-1}((\mathcal{L}_{(-p-1)}/\mathcal{L}_{(-p-2)})^{\#}).
\]
}

\begin{definition}
[4]

 For each integer $q$ satisfying $-d \leq q \leq d$ , the $q^{th}$ Gersten complex $G_{q}$ is defined to be the $-q^{th}$ line of $E_{1}$ page of the above coniveau spectral sequence
 {\footnotesize
\begin{align*} 
 G_{q}: & \  0 \to \bigoplus_{x \in X_{(d)}}K_{d+q}(O_{X} \ on \ x)  \to \bigoplus_{x \in X_{(d-1)}}K_{d+q-1}(O_{X} \ on \ x) \to \dots  \\
& \to \dots \to   \bigoplus_{x \in X_{(-(q-1))}}K_{1}(O_{X,x} \ on \ x) 
  \xrightarrow{d_{1}^{q-1,-q}} \bigoplus_{x \in X_{(-q)}}K_{0}(O_{X,x} \ on \ x) \\
  &  \xrightarrow{d_{1}^{q,-q}} \dots \to \bigoplus_{x \in X_{(-d)}}K_{q-d}(O_{X,x} \ on \ x) \to 0.
\end{align*}
}

\end{definition}

\begin{theorem}
Balmer's K-theoretic {\it (-q)-boundaries} $B_{-q}(\mathcal{L})$ for $\mathcal{L}=D^{perf}(X)$ agrees with  the image of the differential $d_{1}^{q-1,-q}$
\[
 B_{-q}(\mathcal{L}) = Im(d_{1}^{q-1,-q}).
\]
\end{theorem}

\begin{proof}
The long exact sequence
\[
 \dots \to K_{1}((\mathcal{L}_{(-q+1)}/\mathcal{L}_{(-q)})^{\#}) \xrightarrow{k}  K_{0}(\mathcal{L}_{(-q)}) \xrightarrow{i}  K_{0}(\mathcal{L}_{(-q+1)}) \to \dots
\]
shows that $Ker(i) = Im(k)$. So $B_{-q}(\mathcal{L})= j \circ Ker(i) = j \circ Im(k)$,
where $j: K_{0}(\mathcal{L}_{(-q)}) \to K_{0}((\mathcal{L}_{(-q)}/\mathcal{L}_{(-q-1)})^{\#})$.

The conclusion follows since the differential $d_{1}^{q-1,-q}$ is the composition $d = j \circ k$ 
\[
d_{1}^{q-1,-q}: K_{1}((\mathcal{L}_{(-q+1)}/\mathcal{L}_{(-q)})^{\#}) \xrightarrow{k} K_{0}(\mathcal{L}_{(-q)}) \xrightarrow{j}  K_{0}((\mathcal{L}_{(-q)}/\mathcal{L}_{(-q-1)})^{\#}).
\]
\end{proof}

With the above preparation, we are ready to propose our K-theoretic definitions of Chow groups.

\begin{definition}
With chosen {\it dimension function} on $D^{perf}(X)$, the K-theoretic {\it q-cycles} and K-theoretic {\it rational equivalence} of $(X,O_{X})$, denoted $Z_{q}(D^{perf}(X))$ and $Z_{q,rat}(D^{perf}(X))$ respectively, are defined to be 
\[
  Z_{q}(D^{perf}(X))= Ker(d_{1}^{q,-q})
\]
\[
  Z_{q,rat}(D^{perf}(X))=Im(d_{1}^{q-1,-q}).
\]

The $q^{th}$ K-theoretic Chow group  of $(X,O_{X})$, denoted by $CH_{q}(D^{perf}(X))$, is defined to be 

\[
  CH_{q}(D^{perf}(X))= \dfrac{Z_{q}(D^{perf}(X))}{Z_{q,rat}(D^{perf}(X))}.
\]
\end{definition}

It is clear that our K-theoretic Chow groups are cohomology groups of Gersten complexes. It is also clear that our K-theoretic Chow groups are subgroups of Balmer's.
\begin{corollary}
Let $\mathcal{L} = D^{perf}(X)$, we have the following :
\[
  Z_{q,rat}(D^{perf}(X)) = B_{-q}(\mathcal{L})
\]
\[
  Z_{q}(D^{perf}(X)) \subseteq Z_{-q}(\mathcal{L})
\]
\[
  CH_{q}(D^{perf}(X)) \subseteq CH_{-q}(\mathcal{L})
\]
\end{corollary}

\subsection{Agreement}
\label{Agreement}
We show that our K-theoretic Chow groups recover the classical ones for regular schemes. 
\begin{theorem} 
 Let $X$ be a regular scheme of finite type over a field $k$ and let the tensor triangulated category $D^{perf}(X)$ be equipped with $-codim_{Krull}$ as a dimension function, our K-theoretic Chow group agrees with the classical one
 \[
  Z_{q}(D^{perf}(X))= Z^{q}(X)
\]

\[
 Z_{q,rat}(D^{perf}(X)) = Z^{q}_{rat}(X)
\]
 
\[
  CH_{q}(D^{perf}(X)) = CH^{q}(X).
\]
Moreover, we have also Bloch's formula:
 \[
   CH_{q}(D^{perf}(X)) = H^{q}(X, K_{q}(O_{X}))=CH^{q}(X).
 \]
\end{theorem}

\begin{proof}
When the tensor triangulated category $D^{perf}(X)$ is equipped with $-codim_{Krull}$ as a dimension function, $X_{(-i)}=\{ x \in X \mid -codim_{Krull} \overline{\{x\}}= -i\}=\{ x \in X \mid codim_{Krull} \overline{\{x\}}= i\}=X^{(i)}$. So the Gersten complex $G_{q}$ is 
{\footnotesize
\begin{align*}
 0 \to & K_{q}(X) \to \bigoplus_{x \in X^{(0)}}K_{q}(O_{X,x}) \to \dots \xrightarrow{d_{1}^{q-1,-q}} \bigoplus_{x \in X^{(q)}}K_{0}(O_{X,x} \ on \ x) \\
 & \xrightarrow{d_{1}^{q,-q}} \bigoplus_{x \in X^{(q+1)}}K_{-1}(O_{X,x} \ on \ x) \to
 \dots \to \bigoplus_{x \in X^{(d)}}K_{q-d}(O_{X,x} \ on \ x) \to 0.
\end{align*}
}

Since $X$ is a regular scheme of finite type over a field $k$, then Quillen's d\'evissage [25] says the above Gersten complex agrees with Quillen's classical one, whose sheafification is a flasque resolution:
{\footnotesize
\begin{align*}
 0 \to & K_{q}(X) \to K_{q}(k(X)) \to \bigoplus_{x \in X^{(1)}}K_{q-1}(k(x)) \to 
 \dots \xrightarrow{d_{1}^{q-1,-q}} \bigoplus_{x \in X^{(q)}}K_{0}(k(x)) \\
 & \xrightarrow{d_{1}^{q,-q}} \bigoplus_{x \in X^{(q+1)}}K_{-1}(k(x)) \to \dots \to \bigoplus_{x \in X^{(d)}}K_{q-d}(k(x)) \to 0.
\end{align*}
}

Since $k(x)$ is regular, $K_{-1}(k(x))=0$. So 
\[
  Z_{q}(D^{perf}(X))= Ker(d_{1}^{q,-q})= \bigoplus_{x \in X^{(q)}}K_{0}(k(x)) = Z^{q}(X).
\]

As explained in [25],  the image of $d_{1}^{q-1,-q}$ gives the rational equivalence. Hence, 
\[
 Z_{q,rat}(D^{perf}(X)) = Im(d_{1}^{q-1,-q}) = Z^{q}_{rat}(X).
\]

Therefore, we have the following identification
\[
  CH_{q}(D^{perf}(X)) = CH^{q}(X).
\]

Bloch's formula
 \[
   CH_{q}(D^{perf}(X)) = H^{q}(X, K_{q}(O_{X}))=CH^{q}(X)
 \]
 follows immediately from the fact that the sheafification of the Gersten complex is a flasque resolution.
\end{proof}

\begin{theorem} 
 Let $X$ be a $d$-dimensional regular scheme of finite type over a field $k$ and let the tensor triangulated category $D^{perf}(X)$ be equipped with $dim_{Krull}$ as a dimension function, then we have the following identifications 
 \[
  Z_{q}(D^{perf}(X))= Z^{d+q}(X)
\]

\[
 Z_{q,rat}(D^{perf}(X)) = Z^{d+q}_{rat}(X)
\]
 
\[
  CH_{q}(D^{perf}(X)) = CH^{d+q}(X).
\]
\end{theorem}

\begin{proof}
When the tensor triangulated category $D^{perf}(X)$ is equipped with $dim_{Krull}$ as a dimension function, the augmented Gersten complex $G_{q}$ is
\begin{align*}
 & G_{q}:  \  0 \to  K_{d+q}(X) \to  \bigoplus_{x \in X_{(d)}}K_{d+q}(O_{X,x} \ on \ x) \to \dots \to 
  \bigoplus_{x \in X_{(-(q-1))}}K_{1}(O_{X,x} \ on \ x) \\
 & \xrightarrow{d_{1}^{q-1,-q}}  \bigoplus_{x \in X_{(-q)}}K_{0}(O_{X,x} \ on \ x) 
  \xrightarrow{d_{1}^{q,-q}} \bigoplus_{x \in X_{(-(q+1))}}K_{-1}(O_{X,x} \ on \ x) \to
 \dots 
\end{align*}

Since $X_{(-i)}=\{ x \in X \mid dim_{Krull} \overline{\{x\}}= -i\}=\{ x \in X \mid codim_{Krull} \overline{\{x\}}= d+i\}=X^{(d+i)}$, so $G_{q}$ is 
\begin{align*}
 & G_{q}:  \  0 \to  K_{d+q}(X) \to  \bigoplus_{x \in X^{(0)}}K_{d+q}(O_{X,x} \ on \ x) \to \dots \to 
  \bigoplus_{x \in X^{(d+(q-1))}}K_{1}(O_{X,x} \ on \ x)  \\
 & \xrightarrow{d_{1}^{q-1,-q}} \bigoplus_{x \in X^{(d+q)}}K_{0}(O_{X,x} \ on \ x) 
  \xrightarrow{d_{1}^{q,-q}} \bigoplus_{x \in X^{(d+(q+1))}}K_{-1}(O_{X,x} \ on \ x) \to
 \dots 
\end{align*}

The remaining proof is the same as the above one.

\end{proof}

\subsection{Functoriality}
\label{Functoriality}
In this subsection, we discuss functoriality of K-theoretic Chow groups, flat pull-back and proper push-forward.

\textbf{Flat pull-back}.  Let $X$ and $Y$ be noetherian schemes with an ample family of line bundles(([SGA 6] II 2.2.3, or [32] 2.1.1]) and $f : X \to Y$ a flat morphism. 
For a noetherian scheme with an ample family of line bundles $S$($S=X,Y$), $D^{perf}(S)$ is equivalent to the derived category obtained from strict perfect complexes of $O_{S}$-modules(see [32, lemma 3.8]). We use the later in the following and also assume $D^{perf}(S)$ equipped with the dimension
function $-codim_{Krull}$.

\begin{lemma}([SGA 6] 1.2, [32] 2.5.1])
Let $f: X \to Y$ be a map of schemes, $Lf^{\ast}$ sends (strict)perfect complexes
to (strict)perfect complexes
\[
  Lf^{\ast}: D^{perf}(Y) \to D^{perf}(X).
\]
\end{lemma}
 
 \begin{proof}
 For $E^{\bullet}$ a
strict perfect complex on $Y$, $f^{\ast}E^{\bullet}$ is clearly a strict perfect complex on
$X$. This complex represents $Lf^{\ast}E^{\bullet}$ as the vector bundles $E^{i}$ are flat over
$O_{Y}$ and hence deployed for $Lf^{\ast}$.
 \end{proof}
 
Moreover, for $f : X \to Y$ a flat morphism, $Lf^{\ast}$ respects the filtration of dimension of support. This has been proved by Klein in [19, lemma 4.2.1] for regular schemes $X$ and $Y$. In fact , Klein's proof also works in our setting.

\begin{lemma}[19, lemma 4.2.1]

The functor $Lf^{\ast} :  D^{perf}(Y) \to D^{perf}(X)$ respect the filtration of dimension of support

\[
  Lf^{\ast} :  D^{perf}(Y)_{(p)} \to D^{perf}(X)_{(p)}.
\]
 \end{lemma}

\begin{theorem}
Let $f : X \to Y$ be a flat morphism, then $Lf^{\ast}$ induces group homomorphisms

\[
  CH(Lf^{\ast}): CH_{p}(D^{perf}(Y)) \to CH_{p}(D^{perf}(X)).
\]
\end{theorem}

\begin{proof}
$Lf^{\ast}$ respects the filtration of dimension of support
\[
  Lf^{\ast} :  D^{perf}(Y)_{(p)} \to D^{perf}(X)_{(p)}
\]
\[
  Lf^{\ast} :  D^{perf}(Y)_{(p-1)} \to D^{perf}(X)_{(p-1)}.
\]

According to universal property of Verdier quotient, we have 
\[
 Lf^{\ast} :  D^{perf}(Y)_{(p)}/D^{perf}(Y)_{(p-1)} \to D^{perf}(X)_{(p)}/D^{perf}(X)_{(p-1)}.
\]

Furthermore, according to [7], we have 
\[
 Lf^{\ast} :  (D^{perf}(Y)_{(p)}/D^{perf}(Y)_{(p-1)})^{\#} \to (D^{perf}(X)_{(p)}/D^{perf}(X)_{(p-1)})^{\#}.
\]
This induces maps between coniveau spectral sequences $E_{1}^{-p,-q}(X)$ and $E_{1}^{-p,-q}(Y)$
{\footnotesize
\[
 Lf^{\ast}: K_{p+q}((D^{perf}(Y)_{(p)}/D^{perf}(Y)_{(p-1)})^{\#}) \to  K_{p+q}((D^{perf}(X)_{(p)}/D^{perf}(X)_{(p-1)})^{\#}) . 
\]
}

Therefore, we have the following commutative diagram
{\footnotesize
\[
  \begin{CD}
   \dots @>>> \bigoplus_{y \in Y^{(p-1)}}K_{1}(O_{Y,y} \ on \ y) @>d_{1}^{p-1,-p}>> \bigoplus_{y \in Y^{(p)}}K_{0}(O_{Y,y} \ on \ y) @>d_{1}^{p,-p}>> \dots \\
   @VLf^{\ast}VV  @VLf^{\ast}VV  @VLf^{\ast}VV  \\
  \dots @>>> \bigoplus_{x \in X^{(p-1)}}K_{0}(O_{X,x} \ on \ x) @>d_{1}^{p-1,-p}>> \bigoplus_{x \in X^{(p)}}K_{0}(O_{X,x} \ on \ x) @>d_{1}^{p,-p}>> \dots \\
  \end{CD}
\]
}

Hence, $Lf^{\ast}$ induces group homomorphisms
\[
  CH(Lf^{\ast}): CH_{p}(D^{perf}(Y)) \to CH_{p}(D^{perf}(X)).
\]
\end{proof}

\textbf{II: proper push-forward}. Let $X$ and $Y$ still be noetherian schemes with an ample family of line bundles. We assume $D^{perf}(X)$ and $D^{perf}(Y)$ equipped with the dimension
function $dim_{Krull}$. 

Let's recall a theorem from SGA6 firstly. For the definition of { \it pseudo-coherent} and {\it perfect}, 
we refer to [SGA6] III or [32,section 2].
\begin{lemma}([SGA6] III 2.5, 4.8.1 or [32] theorem 2.5.4).
Let $f : X \to Y$ be a
proper map of schemes. Suppose either that $f$ is projective, or that $Y$ is
locally noetherian. Suppose that $f$ is a pseudo-coherent (respectively, a
perfect) map. Then if $E'$ is a pseudo-coherent (resp. perfect) complex on
$X$, $Rf_{\ast}(E')$ is pseudo-coherent (resp. perfect) on $Y$.
\end{lemma}

Examples of perfect maps are smooth maps, regular closed immersion, locally complete intersection ect. We refer the readers to [SGA6] VII for more discussions.
\begin{corollary}
Let $f: X \to Y$ be a proper morphism between noetherian schemes.
Suppose that $f$ is a perfect map. Then $Rf_{\ast}$ sends perfect complexes to perfect complexes
\[
  Rf_{\ast}: D^{perf}(Y) \to D^{perf}(X).
\]
\end{corollary}

We expect that  $Rf_{\ast}$ act like $Lf^{\ast}$,  respecting the filtration by dimension of support. In fact, it does if we allow some assumptions. The following lemma has been proved by Klein in [19]
for both $X$ and $Y$ integral, non-singular, separated schemes of finite type over an algebraically closed field. His proof also works in our setting.

\begin{lemma}[19, lemma 4.3.1]
Let $f: X \to Y$ be a proper morphism between algebraic varieties defined over an algebraically closed field. Suppose that $f$ is a perfect map(so the above corollary applies).
The functor $Rf_{\ast} : D^{perf}(X) \to D^{perf}(Y)$ respects the filtration by dimension of support
\[
  Rf_{\ast} :  D^{perf}(X)_{(p)} \to D^{perf}(Y)_{(p)}.
\]
\end{lemma}

Consequently, we have the following theorem
\begin{theorem} 
With the above assumption,  $Rf_{\ast}$ induces group homomorphisms

\[
  CH(Rf_{\ast}): CH_{p}(D^{perf}(X)) \to CH_{p}(D^{perf}(Y)).
\]
\end{theorem}

\begin{proof}
Similiar to the above theorem 3.12.
\end{proof}

\subsection{Relative Chow groups and Gersten conjecture}
\label{Relative Chow groups and Gersten conjecture}
In [23], Pedrini-Weibel defined the relative Chow group $CH^{q}(X,Y)$. To be precise, 
\begin{definition}
Let $X$ be a connected $d$-dimensional quasi-projective variety($d \geqslant 2$) whose singular locus $sing(X)$ is finite and is contained in a finite closed set $Y$. The relative Chow group $CH^{q}(X,Y)$ was defined to be the cokernel of the cycle map:
\[
 \coprod_{\substack{codim(Z)=q-1 \\ Z\cap Y = \emptyset}}k(Z)^{\ast} \to \coprod_{\substack{codim(Z)=q \\ Z\cap Y = \emptyset}}\mathbb{Z}
\]
\end{definition}
Moreover, Pedrini-Weibel showed that Bloch's formula holds true for relative Chow group $CH^{q}(X,Y)$.
\begin{theorem}[24]

Let $X$ be a connected $d$-dimensional quasi-projective variety($d \geqslant 2$) whose singular locus $sing(X)$ is contained in a finite closed set $Y$. Then for $2\leqslant q \leqslant d$, there are maps
\[
 \beta^{q}: CH^{q}(X,Y) \to H^{q}(X, \underline{K}_{q}(O_{X}))
\]
which are isomorphisms modulo $(d-1)!$ torsion. 
\end{theorem}

It is very interesting to compare our $CH_{q}(D^{perf}(X))$ with Pedrini-Weibel's relative Chow groups $CH^{q}(X,Y)$. This might be closely related with the Gersten conjecture. So we suggest the following question.
\begin{question}
Let $X$ be a connected $d$-dimensional quasi-projective variety($d \geqslant 2$) whose singular locus $sing(X)$ is contained in a finite closed set $Y$. For $2\leqslant q \leqslant d$, after tensoring with $\mathbb{Q}$, is the sheafification of the following Gersten complex a flasque resolution ?
\[
 0 \to K_{q}(X) \to K_{q}(k(X)) \to 
 \dots \to \bigoplus_{x \in X^{(d)}}K_{q-d}(O_{X,x} \ on \ x) \to 0
\]
\end{question}

\begin{remark}
The answer to the above question is negative in general. There exist examples disproving it, e.g Proposition 9 in [4]. Thanks B.Totaro and C.Pedrini for pointing out this.
\end{remark}

\begin{corollary}
When the above question 3.19 holds true with some assumptions, tensoring with $\mathbb{Q}$, then our Chow groups $CH_{q}(D^{perf}(X))$ agree with Pedrini-Weibel's relative Chow groups $CH^{q}(X,Y)$. 
\end{corollary}

\section{Milnor K-theoretic Chow groups of derived categories of regular schemes and thickenings}
\label{Milnor K-theoretic Chow groups of derived categories of regular schemes and thickenings}

\subsection{Definition}
\label{Definition}
Keeping Soul\'e's variant of Bloch's formula in mind, we would like to define Milnor K-theoretic Chow groups of derived categories of schemes. However, we don't have Milnor K-theory with support $K_{m}^{M}(O_{X,x} \ on \ x)$ directly. The following theorem of Soul\'e [29] suggests that we can use suitable eigen-space of Adams' operations to replace Milnor K-theory, ignoring torsion.
\begin{theorem}[Soul\'e, 29]

 For $X$ a regular scheme of finite type over a field $k$ with characteristic $0$, let $\underline{K}^{M}_{m}(O_{X})$(resp. $\underline{K}^{(m)}_{m}(O_{X})$) denote the 
sheaf associated to the presheaf
\[
  U \to K^{M}_{m}(O_{X}(U))
\]
(resp. $K^{(m)}_{m}(O_{X}(U))$), where $K_{m}^{(m)}$ is the eigen-space of $\psi^{k}=k^{m}$(Adams' operations $\psi^{k}$ is recalled in section 4.3).

After ignore torsion, we have the following identification
\[
 \underline{K}^{(m)}_{m}(O_{X}) = \underline{K}^{M}_{m}(O_{X}).
\]
\end{theorem}

Following Soul\'e 's theorem, we define Milnor K-theory with support $K_{m}^{M}(O_{X,x} \ on \ x)$ to be 
suitable eigen-space of $K_{m}(O_{X,x} \ on \ x)$. 
\begin{definition}
Let $X$ be a $d$-dimensional noetherian scheme and $x \in X$ satisfy $dimO_{X,x} = j$. After tensoring with $\mathbb{Q}$, Milnor K-theory with support $K_{m}^{M}(O_{X,x} \ on \ x)$ is defined to be 
\[
  K_{m}^{M}(O_{X,x} \ on \ x) := K_{m}^{(m+j)}(O_{X,x} \ on \ x),
\] 
where $K_{m}^{(m+j)}$ is the eigen-space of $\psi^{k}=k^{m+j}$.
\end{definition}

The reason why we choose $K_{m}^{(m+j)}$ to define $K_{m}^{M}$ is inspired by another theorem of Soul\'e, \textbf{Riemann-Roch without denominator}. 
\begin{theorem}
 Riemann-Roch without denominator--[Soul\'e, 29]

For $X$ regular scheme of finite type and $\eta \in X^{(j)}$, we have(for any integer $m$ and $i$)
\[
 K_{m}^{(i)}(O_{X,\eta} \ on \ \eta) = K_{m}^{(i-j)}(k(\eta)).
\]
\end{theorem}

This theorem says that our definition of Milnor K-theory with support is a honest generalization of the classical one, at least for regular case.

Next, we would like to define Milnor K-theoretic Chow groups of derived categories of schemes by mimicking definition 3.6. In order to do that, we need to detect whether the differentials of the Gersten complex respect Adams' operations. 

To fix the idea, we assume the tensor triangulated category $D^{perf}(X)$ is equipped with $-codim_{Krull}$ as a dimension function in the following. If the differentials $d_{1}^{p,-q}$ of the Gersten complex(definition 3.4) respect Adams' operations, for every $i \in \mathbb{Z}$, then there exists the following refiner (augmented)complex  
{\footnotesize
\begin{align*}
 0 \to & K_{q}^{(i)}(X) \to \dots \to \bigoplus_{x \in X^{(q-1)}}K_{1}^{(i)}(O_{X,x} \ on \ x) 
   \xrightarrow{d_{1,i}^{q-1,-q}} \bigoplus_{x \in X^{(i)}}K_{0}^{(i)}(O_{X,x} \ on \ x) \\  &\xrightarrow{d_{1,i}^{q,-q}} \bigoplus_{x \in X^{(q+1)}}K_{-1}^{(i)}(O_{X,x} \ on \ x) \to \dots \to \bigoplus_{x \in X^{(d)}}K_{q-d}^{(i)}(O_{X,x} \ on \ x) \to 0.
\end{align*}
}

We are interested particularly in the ``Milnor" part. One obtains the following refiner complex by taking $i=q$
{\footnotesize
\begin{align*}
 0 \to & K_{q}^{(q)}(X) \to \dots \to \bigoplus_{x \in X^{(q-1)}}K_{1}^{(q)}(O_{X,x} \ on \ x) 
   \xrightarrow{d_{1,q}^{q-1,-q}} \bigoplus_{x \in X^{(q)}}K_{0}^{(q)}(O_{X,x} \ on \ x) \\  &\xrightarrow{d_{1,q}^{q,-q}} \bigoplus_{x \in X^{(q+1)}}K_{-1}^{(q)}(O_{X,x} \ on \ x) \to 
 \dots \to \bigoplus_{x \in X^{(d)}}K_{q-d}^{M}(O_{X,x} \ on \ x) \to 0.
\end{align*}
}

After tensoring with $\mathbb{Q}$, this complex can be written as 
{\footnotesize
\begin{align*}
 0 \to & K_{q}^{M}(X) \to \dots \to \bigoplus_{x \in X^{(q-1)}}K_{1}^{M}(O_{X,x} \ on \ x) 
   \xrightarrow{d_{1,q}^{q-1,-q}} \bigoplus_{x \in X^{(q)}}K_{0}^{M}(O_{X,x} \ on \ x) \\ &\xrightarrow{d_{1,q}^{q,-q}} \bigoplus_{x \in X^{(q+1)}}K_{-1}^{M}(O_{X,x} \ on \ x) \to \dots
   \dots \to \bigoplus_{x \in X^{(d)}}K_{q-d}^{(q)}(O_{X,x} \ on \ x) \to 0.
\end{align*}
}

\begin{definition}
If the differentials $d_{1}^{p,-q}$ of the Gersten complex respect Adams' operations, then 
the Milnor  K-theoretic {\it q-cycles} and Milnor K-theoretic {\it rational equivalence} of $(X,O_{X})$, denoted $Z^{M}_{q}(D^{perf}(X))$ and $Z^{M}_{q,rat}(D^{perf}(X))$ respectively, are defined to be 
\[
  Z^{M}_{q}(D^{perf}(X))= Ker(d_{1,q}^{q,-q}),
\]
\[
  Z^{M}_{q,rat}(D^{perf}(X))=Im(d_{1,q}^{q-1,-q}).
\]

The $q^{th}$ Milnor K-theoretic Chow group  of $(X,O_{X})$, denoted by $CH^{M}_{q}(D^{perf}(X))$, is defined to be 

\[
  CH^{M}_{q}(D^{perf}(X))= \dfrac{Z^{M}_{q}(D^{perf}(X))}{Z^{M}_{q,rat}(D^{perf}(X))}.
\]
\end{definition}

We shall discuss two cases where the above definition works. One is {\it points on schemes} which don't have negative K-groups because of dimension reason. So we can use Adams' operations(at space level) to refine the Gersten complex. This is discussed in the appendix.

The second one is regular schemes and their thickenings discussed in the rest of this section. We do explicit computation on eigen-spaces of relative (negative) cyclic homology in 4.2.  Goodwillie and Cathelineau type results are proved in 4.3 where Adams' operations on K-groups are briefly recalled.

In 4.4, we show that Milnor K-theoretic Chow groups of regular schemes and their thickenings are well-defined. Moreover, we show Milnor K-theoretic Chow groups satisfy Bloch's formula. This gives a positive answer to Green-Griffiths' question on extending Bloch-Quillen identification.

\subsection{Adams' operations on negative cyclic homology}
\label{Adams' operations on negative cyclic homology}
In this subsection, we do explicit computation on Adams' eigen-spaces of relative (negative) cyclic homology for later use.

Let's begin with recalling notations and definitions. Let $A$ be any commutative $k$-algebra, where $k$ is a field of characteristic $0$, and $I$ be an ideal of $A$.
One can associate a Hochschild complexes $C^{h}_{\ast}(A)$ to $A$ which has $C_{n}^{h}(A)=A^{\otimes n+1}$:
\[
 C_{\ast}^{h} : \dots \xrightarrow{b} A^{\otimes \ast+1} \xrightarrow{b} \dots \xrightarrow{b} A\otimes A \xrightarrow{b} A \to 0.
\]

 The action of the symmetric groups on $C^{h}_{\ast}(A)$
gives the lambda operation
\[
 HH_{n}(A)=HH_{n}^{(1)}(A) \oplus \dots \oplus HH_{n}^{(n)}(A),
\]
and similarly
\[
 HC_{n}(A)=HC_{n}^{(1)}(A) \oplus \dots \oplus HC_{n}^{(n)}(A),
\]
\[
 HN_{n}(A)=HN_{n}^{(1)}(A) \oplus \dots \oplus HN_{n}^{(n)}(A).
\]

 There is also a Hochschild complexes  
$C^{h}_{\ast}(A/I)$ associated to  $A/I$. We use $C^{h}_{\ast}(A,I)$ to denote the kernel of the natural map
\[
 C^{h}_{\ast}(A) \rightarrow C^{h}_{\ast}(A/I).
\]
Then the relative Hochschild module $HH_{\ast}(A,I)$ is the homology of the complex $C^{h}_{\ast}(A,I)$. Moreover, the action of the symmetric groups on $C^{h}_{\ast}(A,I)$
gives the lambda operation
\[
 HH_{n}(A,I)=HH_{n}^{(1)}(A,I) \oplus \dots \oplus HH_{n}^{(n)}(A,I)
\]
and similarly
\[
 HC_{n}(A,I)=HC_{n}^{(1)}(A,I) \oplus \dots \oplus HC_{n}^{(n)}(A,I),
\]
\[
 HN_{n}(A,I)=HN_{n}^{(1)}(A,I) \oplus \dots \oplus HN_{n}^{(n)}(A,I).
\]

From now on,  $R$ is a regular Noetherian domain and commutative $\mathbb{Q}$-algebra, and $\varepsilon$ is the dual number. We consider $R[\varepsilon]= R \oplus \varepsilon R $ as a graded $\mathbb{Q}$-algebra.
The following SBI sequence is obtained from the corresponding eigen-piece of the relative Hochschild complex:
{\small
\[
 \rightarrow HC^{(i)}_{n+1}(R[\varepsilon],\varepsilon) \xrightarrow{S} HC^{(i-1)}_{n-1}(R[\varepsilon],\varepsilon) \xrightarrow{B} HH^{(i)}_{n}(R[\varepsilon],\varepsilon) \xrightarrow{I} HC^{(i)}_{n}(R[\varepsilon],\varepsilon) \rightarrow
\]
}
According to Geller-Weibel [17], the above S map is $0$ on $HC(R[\varepsilon],\varepsilon)$. This enable us to break the SBI sequence up into 
short exact sequence:
\[
 0 \rightarrow HC^{(i-1)}_{n-1}(R[\varepsilon],\varepsilon) \xrightarrow{B} HH^{(i)}_{n}(R[\varepsilon],\varepsilon) \xrightarrow{I} HC^{(i)}_{n}(R[\varepsilon],\varepsilon) \rightarrow 0
\]

In the following, we will use this short exact sequence to compute $HC^{(i)}_{n}(R[\varepsilon],\varepsilon)$.
\begin{theorem}
\begin{equation}
\begin{cases}
 \begin{CD}
 HC_{n}^{(i)}(R[\varepsilon],\varepsilon)=\Omega^{{2i-n}}_{R/ \mathbb{Q}}, \ for \  \frac{n}{2} \leq i \leq n.\\
  HC_{n}^{(i)}(R[\varepsilon],\varepsilon)=0, \ else.
 \end{CD}
\end{cases}
\end{equation} 

\end{theorem}

\begin{proof}
 Step 1. we will prove 
\[
 HC_{n}^{(i)}(R[\varepsilon],\varepsilon)= 0,  for \ i < \frac{n}{2}.
\]
by showing $HH^{(i)}_{n}(R[\varepsilon],\varepsilon)=0$. Noting that $HH^{(i)}_{n}(R)=0$, it suffices to show $HH^{(i)}_{n}(R[\varepsilon])=0$, for $i < \frac{n}{2}$.
By applying K\"{u}nneth formula to $R[\varepsilon]= R \otimes_{\mathbb{Q}} \mathbb{Q}[\varepsilon] $, we have 
\begin{align*}
 HH^{(i)}_{n}(R[\varepsilon])=& HH^{(0)}_{0}(R) \otimes HH^{(i)}_{n}(\mathbb{Q}[\varepsilon]) \oplus HH^{(1)}_{1}(R) \otimes HH^{(i-1)}_{n-1}(\mathbb{Q}[\varepsilon]) \oplus \\
 & \dots \oplus HH^{(i)}_{i}(R) \otimes HH^{(0)}_{n-i}(\mathbb{Q}[\varepsilon]).
\end{align*}
According to [22, 5.4.15], the only possibilities for $HH^{(i-j)}_{n-j}(\mathbb{Q}[\varepsilon])$ being nonzero are the followings:

\begin{equation}
\begin{cases}
 \begin{CD}
 n-j \ is \ even,  \ n-j=2(i-j).\\
 n-j \ is \ odd,  \ n-j + 1=2(i-j).
 \end{CD}
\end{cases}
\end{equation} 

Neither of them will occur, since $i < \frac{n}{2}$. Therefore, $HH^{(i)}_{n}(R[\varepsilon])=0$.

Step 2. we will show that
\[
 HC_{n}^{(i)}(R[\varepsilon],\varepsilon)= \Omega^{{2i-n}}_{R/ \mathbb{Q}}, for \ \frac{n}{2} \leq i < n.
\]
by computing $HH^{(i)}_{n}(R[\varepsilon],\varepsilon)$ directly and using induction on $HC^{(i-1)}_{n-1}(R[\varepsilon],\varepsilon)$.

Firstly, we have  $HH^{(i)}_{n}(R)=0$ and  $HH^{(i)}_{n}(R[\varepsilon])$ can be expressed as 
\begin{align*}
 HH^{(i)}_{n}(R[\varepsilon])= &HH^{(0)}_{0}(R) \otimes HH^{(i)}_{n}(\mathbb{Q}[\varepsilon]) \oplus HH^{(1)}_{1}(R) \otimes HH^{(i-1)}_{n-1}(\mathbb{Q}[\varepsilon]) \oplus \\
 & \dots \oplus HH^{(i)}_{i}(R) \otimes HH^{(0)}_{n-i}(\mathbb{Q}[\varepsilon]).
\end{align*}
According to [22, 5.4.15], the only possibilities for $HH^{(i-j)}_{n-j}(\mathbb{Q}[\varepsilon])$ being nonzero are the followings:
\begin{equation}
\begin{cases}
 \begin{CD}
 n-j \ is \ even, \ n-j=2(i-j), \ then \ j= 2i-n.\\
 n-j \ is \ odd, \ n-j + 1=2(i-j), \ then \ j= 2i-n-1.
 \end{CD}
\end{cases}
\end{equation} 

Therefore,
{\small
\[
 HH^{(i)}_{n}(R[\varepsilon])= HH^{(2i-n)}_{2i-n}(R) \otimes HH^{(n-i)}_{2n-2i}(\mathbb{Q}[\varepsilon]) \oplus HH^{(2i-n-1)}_{2i-n-1}(R) \otimes HH^{(n-i+1)}_{2n-2i+1}(\mathbb{Q}[\varepsilon]).
\]
}
\[
 HH^{(i)}_{n}(R[\varepsilon])= \Omega^{{2i-n}}_{R/ \mathbb{Q}} \oplus \Omega^{{2i-n-1}}_{R/ \mathbb{Q}}.
\]
By induction, 
\[
HC^{(i-1)}_{n-1}(R[\varepsilon],\varepsilon)= \Omega^{{2i-n-1}}_{R/ \mathbb{Q}}.
\]
thus, 
\[
 HC_{n}^{(i)}(R[\varepsilon],\varepsilon)= \Omega^{{2i-n}}_{R/ \mathbb{Q}}, for \ \frac{n}{2} \leq i < n.
\]

Step 3. We prove the formula for $i=n$. It is known that 
\[
 HH^{(n)}_{n}(R)= \Omega^{{n}}_{R/ \mathbb{Q}}.
\]
and 
\begin{align*}
 HH^{(n)}_{n}(R[\varepsilon])= & HH^{(0)}_{0}(R) \otimes HH^{(n)}_{n}(\mathbb{Q}[\varepsilon]) \oplus HH^{(1)}_{1}(R) \otimes HH^{(n-1)}_{n-1}(\mathbb{Q}[\varepsilon]) \oplus \\
 & \dots \oplus HH^{(n)}_{n}(R) \otimes HH^{(0)}_{0}(\mathbb{Q}[\varepsilon]).
\end{align*}

Since $HH^{(i)}_{i}(\mathbb{Q}[\varepsilon])=0$,unless $i=0,1$, we have 
\[
 HH^{(n)}_{n}(R[\varepsilon])= HH^{(n)}_{n}(R) \otimes HH^{(0)}_{0}(\mathbb{Q}[\varepsilon]) \oplus HH^{(n-1)}_{n-1}(R) \otimes HH^{(1)}_{1}(\mathbb{Q}[\varepsilon]),
\]
which can be simplified as 
\[
 HH^{(n)}_{n}(R[\varepsilon])= \Omega^{{n}}_{R/ \mathbb{Q}} \otimes \mathbb{Q}[\varepsilon] \oplus \Omega^{{n-1}}_{R/ \mathbb{Q}} \otimes \mathbb{Q}.
\]
 Therefore, we have
\[
 HH^{(n)}_{n}(R[\varepsilon],\varepsilon) = \Omega^{{n}}_{R/ \mathbb{Q}} \oplus \Omega^{{n-1}}_{R/ \mathbb{Q}}.
\]
Once again, we still have 
\[
HC^{(n)}_{n}(R[\varepsilon],\varepsilon) = \Omega^{{n}}_{R/ \mathbb{Q}}.
\]
\end{proof}
The above result tells us that
\begin{theorem}
\[
 HC_{n}(R[\varepsilon],\varepsilon) = \Omega^{{n}}_{R/ \mathbb{Q}}\oplus \Omega^{{n-2}}_{R/ \mathbb{Q}} \oplus \dots
\]
the last term is $\Omega^{{1}}_{R/ \mathbb{Q}}$ or $R$, depending on $n$ odd or even.
\end{theorem}
The following corollaries are obvious from the fact that for any commutative $k$-algebra  $A$ , where $k$ is a field of characteristic $0$, and $I$ be an ideal of $A$,
\[
 HN_{n}(A,I)=HC_{n-1}(A,I).
\]
\[
 HN_{n}^{(i)}(A,I)=HC_{n-1}^{(i-1)}(A,I).
\]

\begin{corollary}
\begin{equation}
\begin{cases}
 \begin{CD}
 HN_{n}^{(i)}(R[\varepsilon],\varepsilon)= \Omega^{{2i-n-1}}_{R/ \mathbb{Q}}, for \ \frac{n}{2} < i \leq n.\\
 HN_{n}^{(i)}(R[\varepsilon],\varepsilon)= 0, else.
 \end{CD}
\end{cases}
\end{equation} 
 
\end{corollary}

\begin{corollary}
\[
 HN_{n}(R[\varepsilon],\varepsilon) = \Omega^{{n-1}}_{R/ \mathbb{Q}}\oplus \Omega^{{n-3}}_{R/ \mathbb{Q}} \oplus \dots
\]
the last term is $\Omega^{{1}}_{R/ \mathbb{Q}}$ or $R$, depending on $n$ even or odd.
\end{corollary}

We can also generalize the above results to the sheaf level.
\begin{theorem}
Let $X$ be a regular scheme over a field $k$, $chark=0$. we have the following

\begin{equation}
\begin{cases}
 \begin{CD}
 HN_{n}^{(i)}(O_{X}[\varepsilon],\varepsilon)= \Omega^{{2i-n-1}}_{O_{X}/ \mathbb{Q}}, for \  \frac{n}{2} < i \leq n.\\
 HN_{n}^{(i)}(O_{X}[\varepsilon],\varepsilon)= 0, else.
 \end{CD}
\end{cases}
\end{equation} 

It follows that 
\[
 HN_{n}(O_{X}[\varepsilon],\varepsilon) = \Omega^{{n-1}}_{O_{X}/ \mathbb{Q}}\oplus \Omega^{{n-3}}_{O_{X}/ \mathbb{Q}} \oplus \dots
\]
the last term is $\Omega^{{1}}_{O_{X}/ \mathbb{Q}}$ or $O_{X}$, depending on $n$ even or odd.
\end{theorem}

\subsection{Goodwillie and Cathelineau isomorphism}
\label{Goodwillie and Cathelineau isomorphism}
In this subsection, we will show Goodwillie and Cathelineau type results for non-connective K-groups.
We begin with recalling Adams' operations on K-groups.

\textbf{Adams' operations on K-groups}.
In [29], Soul\'e showed that there exists Adams operations $\psi^{k}$ acting on K-groups with supports $K_{n}( X \ on \ Y)$, $n\geq 0$. He considered K-theory as a generalized cohomology theory:
\[
 K= \mathbb{Z} \times BGL^{+}.
\]

Hence, we have
\[
 K(X) = \mathbb{H}(X,\mathbb{Z} \times BGL^{+})
\]
and
\[
 K(X \ on \ Y) = \mathbb{H}_{Y}(X,\mathbb{Z} \times BGL^{+}).
\]
where $Y$ is closed in $X$ and and $K(X \ on \ Y)$, K-theory of $X$ with support in $Y$, is defined as the homotopy fibre of 
\[
 BQP(X) \rightarrow BQP(X \ - \ Y)
\]
here $P(X)$ is the category of locally free sheaves of finite rank on $X$ and $Q$ stands for Quillen's Q-construction.

Now, we let $R_{\mathbb{Z}}(GL_{N})$ be the Grothendieck group of representations of the general linear group scheme of $GL_{N}$. Then it is well known that $R_{\mathbb{Z}}(GL_{N})$
 has a $\lambda$-ring structure. And moreover, an element of $R_{\mathbb{Z}}(GL_{N})$ induces a morphism
\[
 \mathbb{Z} \times BGL_{N}^{+} \to \mathbb{Z} \times BGL^{+}.
\]

In other word, there is a morphism between abelian groups:
\[
 R_{\mathbb{Z}}(GL_{N}) \rightarrow [\mathbb{Z} \times BGL_{N}^{+}, \mathbb{Z} \times BGL^{+}].
\]
Passing to limit, we have
\[
 R_{\mathbb{Z}}(GL) \rightarrow [\mathbb{Z} \times BGL^{+}, \mathbb{Z} \times BGL^{+}].
\]
Furthermore, we have the following morphism by taking hypercohomology:
\[
 R_{\mathbb{Z}}(GL) \rightarrow [\mathbb{H}_{Y}(X,\mathbb{Z} \times BGL^{+}), \mathbb{H}_{Y}(X,\mathbb{Z} \times BGL^{+})].
\]
And finally we arrive at group level:
\[
 R_{\mathbb{Z}}(GL) \rightarrow [K_{m}(X \ on \ Y), K_{m}(X \ on \ Y)].
\]

In other word, the $\lambda$-operations on $K_{m}(X \ on \ Y)$ are induced from the $\lambda$-operations of $R_{\mathbb{Z}}(GL_{N})$. In fact, this is exact the point to prove 
$K_{m}(X \ on \ Y)$ carries a $\lambda$-ring structure.

Since the appearance of the non-zero negative non-connective K-groups in our study, we need to extend the above Adams operations  $\psi^{k}$ to negative range.
This can be done by descending induction, according to Weibel[34]. 

For every integer $n \in \mathbb{Z}$,  we have the following Bass fundamental exact sequence.
\begin{align*}
 0 \to & K_{n}(X \ on \ Y) \to K_{n}(X[t] \ on \ Y[t]) \oplus K_{n}(X[t^{-1}] \ on \ Y[t^{-1}])  \\
 & \to K_{n}(X[t,t^{-1}] \ on \ Y[t,t^{-1}]) \to  K_{n-1}(X \ on \ Y)  \to 0.
\end{align*}

For any $x \in K_{-1}(X \ on \ Y)$, we have $x\cdot t \in K_{0}(X[t,t^{-1}] \ on \ Y[t,t^{-1}])$, where 
$t \in K_{1}(k[t,t^{-1}])$. We have
\[
 \psi^{k}(x\cdot t ) = \psi^{k}(x)\psi^{k}(t)= \psi^{k}(x) k\cdot t.
\]

Tensoring with $\mathbb{Q}$, we have obtained Adams operations $\psi^{k}$ on $K_{-1}(X \ on \ Y)$:
\[
 \psi^{k}(x)= \dfrac{\psi^{k}(x\cdot t )}{k\cdot t}.
\]

Continuing this procedure, we obtain Adams operations on all the negative K-groups.

\textbf{Goodwillie and Cathelineau isomorphism}.
Now, we show Goodwillie and Cathelineau-type results for non-connective K-groups. Let's recall that in [13] Goodwillie shows the relative 
Chern character is an isomorphism between the relative K-group $K_{n}(A,I)$ and negative cyclic homology $HN_{n}(A,I)$, where $A$ is a commutative $\mathbb{Q}$-algebra and $I$
is a nilpotent ideal in $A$. 
\begin{theorem}[13]
Let $I$ be a nilpotent ideal in a commutative $\mathbb{Q}$-algebra $A$, the relative Chern character 
\[
  Ch: K_{n}(A,I) \to HN_{n}(A,I)
\]
is an isomorphism.
\end{theorem}

This result is further generalized by Cathelineau in [8]
\begin{theorem}[8]
The Goodwillie's isomorphism
\[
  K_{n}(A,I) = HN_{n}(A,I)
\]
respects Adams operation. That is,
\[
 K_{n}^{(i)}(A,I) = HN_{n}^{(i)}(A,I),
\]
here $K_{n}^{(i)}$ and $HN_{n}^{(i)}$ are eigen-spaces of $\psi^{k}=k^{i}$ and $\psi^{k}=k^{i+1}$ respectively.
\end{theorem}

In [10, appendix B], Corti$\tilde{n}$as-Haesemeyer-Weibel show a space level version of Goodwillie's theorems. For every  nilpotent sheaf of ideal $I$, they define $K(O,I)$ and $HN(O,I)$ as the following presheaves respectively:
\[
 U \rightarrow K(O(U),I(U))
\]
and
\[
 U \rightarrow HN(O(U),I(U)).
\]

They write $\mathcal{K}(O,I)$ and $\mathcal{HN}(O,I)$ for the presheaves of spectrum whose initial spaces are $K(O,I)$ and $HN(O,I)$ respectively. 
Moreover, one defines $\mathcal{K}^{(i)}(O,I)$ as the homotopy fiber of $\mathcal{K}(O,I)$ on which $\psi^{k}-k^{i}$ acts acyclicly. And we define $\mathcal{HN}^{(i)}(O,I)$ similarly. 
Goodwillie's theorem and Cathelineau's isomorphism can be generalized in the following way.
\begin{theorem}[10]
 The relative Chern character induces homotopy equivalence of spectra:
\[
 Ch: \mathcal{K}(O,I) \simeq \mathcal{HN}(O,I)
\]
and 
\[
 Ch: \mathcal{K}^{(i)}(O,I) \simeq \mathcal{HN}^{(i)}(O,I).
\]
\end{theorem}

Now, let $X$ be a scheme essenially finite type over a field $k$, where $Char k=0$. Let $Y$ be a closed subset in a scheme $X$ and $U = X - Y$. 

Let $\mathbb{H}(X,\bullet)$ denote Thomason's hypercohomology of spectra. We have the following 
Nine-diagrams(each column and row are homotopy fibration):

\[
  \begin{CD}
     \mathbb{H}_{Y}(X, \mathcal{K}(O, \varepsilon)) @>>> \mathbb{H}(X, \mathcal{K}(O, \varepsilon)) @>>>  \mathbb{H}(U, \mathcal{K}(O, \varepsilon)) \\
     @VVV  @VVV   @VVV  \\
     \mathbb{H}_{Y}(X, \mathcal{K}(O_{X}[\varepsilon])) @>>> \mathbb{H}(X, \mathcal{K}(O_{X}[\varepsilon])) @>>>  \mathbb{H}(U, \mathcal{K}(O_{U}[\varepsilon])) \\
     @VVV  @VVV   @VVV  \\
     \mathbb{H}_{Y}(X, \mathcal{K}(O_{X})) @>>> \mathbb{H}(X, \mathcal{K}(O_{X})) @>>>  \mathbb{H}(U, \mathcal{K}(O_{U})) \\
  \end{CD}
\]
and

\[
  \begin{CD}
     \mathbb{H}_{Y}(X, \mathcal{HN}(O, \varepsilon)) @>>> \mathbb{H}(X, \mathcal{HN}(O, \varepsilon)) @>>>  \mathbb{H}(U, \mathcal{HN}(O, \varepsilon)) \\
     @VVV  @VVV   @VVV  \\
     \mathbb{H}_{Y}(X, \mathcal{HN}(O_{X}[\varepsilon])) @>>> \mathbb{H}(X, \mathcal{HN}(O_{X}[\varepsilon])) @>>>  \mathbb{H}(U, \mathcal{HN}(O_{U}[\varepsilon])) \\
     @VVV  @VVV   @VVV  \\
     \mathbb{H}_{Y}(X, \mathcal{HN}(O_{X})) @>>> \mathbb{H}(X, \mathcal{HN}(O_{X})) @>>>  \mathbb{H}(U, \mathcal{HN}(O_{U})) \\
  \end{CD}
\]

The above diagrams result in the following result
\begin{theorem} 
$\mathbb{H}_{Y}(X, \mathcal{K}(O, \varepsilon))$ is the homotpy fibre of 
\[
 \mathbb{H}_{Y}(X, \mathcal{K}(O_{X}[\varepsilon])) \rightarrow \mathbb{H}_{Y}(X, \mathcal{K}(O_{X})),
\]
and $\mathbb{H}_{Y}(X, \mathcal{HN}(O, \varepsilon))$ is the homotopy fibre of
\[
 \mathbb{H}_{Y}(X, \mathcal{HN}(O_{X}[\varepsilon])) \rightarrow \mathbb{H}_{Y}(X, \mathcal{HN}(O_{X})).
\]
\end{theorem}

Combining Goodwillie's isomorphism(space version) with the above result, we have proved the following theorem, which can be considered as a Goodwillie-type isomorphism for relative 
K-groups with support.
\begin{theorem}
Let $K_{n}(X[\varepsilon] \ on \ Y[\varepsilon], \varepsilon)$ denote the kernel of 
\[
  K_{n}(X[\varepsilon] \ on \ Y[\varepsilon]) \rightarrow K_{n}(X \ on \ Y)
\]
and $HN_{n}(X[\varepsilon] \ on \ Y[\varepsilon], \varepsilon)$ denote the kernel of 
\[
  HN_{n}(X[\varepsilon] \ on \ Y[\varepsilon]) \rightarrow HN_{n}(X \ on \ Y),
\]
we have
\[
 K_{n}(X[\varepsilon] \ on \ Y[\varepsilon], \varepsilon) = HN_{n}(X[\varepsilon] \ on \ Y[\varepsilon], \varepsilon).
\]
\end{theorem}

According to [10, appendix B], there exists  the following two splitting fibrations
\[
 \mathcal{K}^{(i)}(O, \varepsilon) \rightarrow \mathcal{K}(O, \varepsilon) \rightarrow \prod_{j\neq i}\mathcal{K}^{(j)}(O, \varepsilon),
\]
and
\[
 \mathcal{HN}^{(i)}(O, \varepsilon) \rightarrow \mathcal{HN}(O, \varepsilon) \rightarrow \prod_{j\neq i}\mathcal{HN}^{(j)}(O, \varepsilon).
\]

Sine taking $\mathbb{H}_{Y}(X,-)$ perserves homotopy fibrations, there exists the following two splitting fibrations:
 \[
  \mathbb{H}_{Y}(X, \mathcal{K}^{(i)}(O, \varepsilon)) \to \mathbb{H}_{Y}(X, \mathcal{K}(O, \varepsilon))  \xrightarrow{\psi^{k}-k^{i}}    \mathbb{H}_{Y}(X ,\prod_{j\neq i}\mathcal{K}^{(j)}(O, \varepsilon)),
 \]
\[
 \mathbb{H}_{Y}(X, \mathcal{HN}^{(i)}(O, \varepsilon)) \to \mathbb{H}_{Y}(X, \mathcal{HN}(O, \varepsilon))  \xrightarrow{\psi^{k}-k^{i+1}}    \mathbb{H}_{Y}(X,\prod_{j\neq i}\mathcal{HN}^{(j)}(O, \varepsilon )).
\]

Passing to group level, we obtain the following results:
\begin{theorem}
\[
 \mathbb{H}^{-n}_{Y}(X, \mathcal{K}^{(i)}(O, \varepsilon)) = \{x \in \mathbb{H}^{-n}_{Y}(X, \mathcal{K}(O, \varepsilon))| \psi^{k}(x)-k^{i}(x)=0 \}.
\]
\[
 \mathbb{H}^{-n}_{Y}(X, \mathcal{HN}^{(i)}(O, \varepsilon)) = \{x \in \mathbb{H}^{-n}_{Y}(X, \mathcal{HN}(O, \varepsilon))| \psi^{k}(x)-k^{i+1}(x)=0 \}.
\]
\end{theorem}

We have shown that
\[
 \mathbb{H}^{-n}_{Y}(X, \mathcal{K}(O, \varepsilon)) = K_{n}(X[\varepsilon] \ on \ Y[\varepsilon], \varepsilon)
\]
and
\[
 \mathbb{H}^{-n}_{Y}(X, \mathcal{HN}(O, \varepsilon)) = HN_{n}(X[\varepsilon] \ on \ Y[\varepsilon], \varepsilon).
\]
Therefore, the homotopy equivalences
\[ 
  \mathcal{K}(O, \varepsilon) \simeq \mathcal{HN}(O, \varepsilon)
\]
and
\[
 \mathcal{K}^{(i)}(O, \varepsilon) \simeq\mathcal{HN}^{(i)}(O, \varepsilon),
\]
give us the following  refiner result:

\begin{theorem}
\[
 K_{n}^{(i)}(X[\varepsilon] \ on \ Y[\varepsilon],\varepsilon) = HN_{n}^{(i)}(X[\varepsilon] \ on \ Y[\varepsilon],\varepsilon).
\]
\end{theorem}

This result enables us to compute the relative  K-groups with support in terms of the relative negative cyclic groups with support. Now, we show an explicit computation on relative negative cyclic groups with support which will be used later. Recall that $X^{(j)}= \{x \in X \mid dim_{Krull}O_{X,x} = j \}$.

\begin{theorem}
Suppose $X$ is a $d$-dimensional regular scheme over a field $k$, where $Char k=0$ and $y \in X^{(j)}$.  For any integer $m$, we have 
\[
 HN_{m}(O_{X,y}[\varepsilon] \ on \ y[\varepsilon],\varepsilon)= H_{y}^{j}(\Omega^{\bullet}_{O_{X,y}/\mathbb{Q}}),
\]
where $\Omega^{\bullet}_{O_{X,y}/\mathbb{Q}}=\Omega^{m+j-1}_{O_{X,y}/\mathbb{Q}}\oplus \Omega^{m+j-3}_{O_{X,y}/\mathbb{Q}}\oplus \dots$
\end{theorem}

\begin{proof}
$O_{X,y}$ is a regular local ring with dimension $j$, so the depth of $O_{X,y}$ is $j$. For each $n \in \mathbb{Z}$,  $\Omega^{n}_{O_{X,y}/\mathbb{Q}}$ can be written as a direct limit of 
$O_{X,y}^{\oplus}$'s. Therefore, $\Omega^{n}_{O_{X,y}/\mathbb{Q}}$ has depth $j$.

Let $HN_{m}(O_{X,y}[\varepsilon] \ on \ y[\varepsilon],\varepsilon)$ denote the kernel of the projection:
\[
HN_{m}(O_{X,y}[\varepsilon] \ on \ y[\varepsilon])  \xrightarrow{\varepsilon =0} HN_{m}(O_{X,y} \ on \ y).
\]
Then $HN_{m}(O_{X,y}[\varepsilon] \ on \ y[\varepsilon],\varepsilon)$ can be identified with the hypercohomology $\mathbb{H}_{y}^{-m}(O_{X,y},HN(O_{X,y}[\varepsilon],\varepsilon))$,
where $HN(O_{X,y}[\varepsilon],\varepsilon)$ is the relative negative cyclic complex, that is the kernel of
\[
 HN(O_{X,y}[\varepsilon]) \xrightarrow{\varepsilon=0} HN(O_{X,y}).
\]

There is a spectral sequence :
\[
 H_{y}^{p}(O_{X,y}, H^{q}(HN(O_{X,y}[\varepsilon],\varepsilon))) \Longrightarrow \mathbb{H}_{y}^{-m}(HN(O_{X,y}[\varepsilon],\varepsilon)).
\]

By corollary 4.8, we have
 \[
 H^{q}(HN(O_{X,y}[\varepsilon],\varepsilon))= HN_{-q}(O_{X,y}[\varepsilon],\varepsilon)= \Omega^{-q-1}_{O_{X,y}/\mathbb{Q}}\oplus \Omega^{-q-3}_{O_{X,y}/\mathbb{Q}}\oplus \dots
 \]
 As each $\Omega^{n}_{O_{X,y}/\mathbb{Q}}$ has depth $j$, only $H_{y}^{j}(X,H^{q}(HN(O_{X,y}[\varepsilon],\varepsilon)))$ can survive because of the depth condition.
This means $q=-m-j$ and 
\[
 H^{-m-j}(HN(O_{X,y}[\varepsilon],\varepsilon))=HN_{m+j}(O_{X,y}[\varepsilon],\varepsilon)= \Omega^{m+j-1}_{O_{X,y}/\mathbb{Q}}\oplus \Omega^{m+j-3}_{O_{X,y}/\mathbb{Q}}\oplus \dots
\]
Let's write 
\[
\Omega^{\bullet}_{O_{X,y}/\mathbb{Q}} = \Omega^{m+j-1}_{O_{X,y}/\mathbb{Q}}\oplus \Omega^{m+j-3}_{O_{X,y}/\mathbb{Q}}\oplus \dots
\]
Thus 
\[
 \mathbb{H}_{y}^{-m}(HN(O_{X,y}[\varepsilon],\varepsilon))=H_{y}^{j}(\Omega^{\bullet}_{O_{X,y}/\mathbb{Q}}).
\]
this means
\[
 HN_{m}(O_{X,y}[\varepsilon] \ on \ y_{\varepsilon},\varepsilon)= H_{y}^{j}(\Omega^{\bullet}_{O_{X,y}/\mathbb{Q}}).
\]
\end{proof}

Repeating the above proof and noting corollary 4.7, we have the following refiner result:
\begin{theorem}
 Suppose $X$ is a $d$-dimensional regular scheme over a field $k$, where $Char k=0$ and $y \in X^{(j)}$. For any integer $m$, we have 
\[
 HN^{(i)}_{m}(O_{X,y}[\varepsilon] \ on \ y[\varepsilon],\varepsilon)= H_{y}^{j}(\Omega^{\bullet,(i)}_{O_{X,y}/\mathbb{Q}}),
\]
where 
\begin{equation}
\begin{cases}
 \begin{CD}
 \Omega_{O_{X,y}/ \mathbb{Q}}^{\bullet,(i)}= \Omega^{{2i-(m+j)-1}}_{O_{X,y}/ \mathbb{Q}}, for \  \frac{m+j}{2}  < \ i \leq m+j.\\
  \Omega_{O_{X,y}/ \mathbb{Q}}^{\bullet,(i)}= 0, else.
 \end{CD}
\end{cases}
\end{equation} 
\end{theorem}

Combining theorem 4.16 and theorem 4.17 with theorem 4.18, we have the following corollary
\begin{corollary}
Under the same assumption as above, we have
\[
 K_{m}(O_{X,y}[\varepsilon] \ on \ y[\varepsilon],\varepsilon)= H_{y}^{j}(\Omega^{\bullet}_{O_{X,y}/\mathbb{Q}}),
\]
where $\Omega^{\bullet}_{O_{X,y}/\mathbb{Q}}=\Omega^{m+j-1}_{O_{X,y}/\mathbb{Q}}\oplus \Omega^{m+j-3}_{O_{X,y}/\mathbb{Q}}\oplus \dots$

Moreover, we have
\[
 K^{(i)}_{m}(O_{X,y}[\varepsilon] \ on \ y[\varepsilon],\varepsilon)= H_{y}^{j}(\Omega^{\bullet,(i)}_{O_{X,y}/\mathbb{Q}}),
\]
where 
\begin{equation}
\begin{cases}
 \begin{CD}
 \Omega_{O_{X,y}/ \mathbb{Q}}^{\bullet,(i)}= \Omega^{{2i-(m+j)-1}}_{O_{X,y}/ \mathbb{Q}}, for \  \frac{m+j}{2}  < \ i \leq m+j.\\
  \Omega_{O_{X,y}/ \mathbb{Q}}^{\bullet,(i)}= 0, else.
 \end{CD}
\end{cases}
\end{equation} 
\end{corollary}

We have the following generalization for $\varepsilon^{n}=0$, where $n$ is any integer.
\begin{corollary}
Let $\varepsilon$ satisfy $\varepsilon^{n}=0$. Under the same assumption as above, we have
\[
 K_{m}(O_{X,y}[\varepsilon] \ on \ y[\varepsilon],\varepsilon)= H_{y}^{j}(\Omega^{\bullet}_{O_{X,y}/\mathbb{Q}}),
\]
where $\Omega^{\bullet}_{O_{X,y}/\mathbb{Q}}=(\Omega^{m+j-1}_{O_{X,y}/\mathbb{Q}}\oplus \Omega^{m+j-3}_{O_{X,y}/\mathbb{Q}}\oplus \dots)^{\oplus n-1}$.

Moreover, we have 
\[
 K^{(i)}_{m}(O_{X,y}[\varepsilon] \ on \ y[\varepsilon],\varepsilon)= H_{y}^{j}(\Omega^{\bullet,(i)}_{O_{X,y}/\mathbb{Q}}),
\]
where 
\begin{equation}
\begin{cases}
 \begin{CD}
 \Omega_{O_{X,y}/ \mathbb{Q}}^{\bullet,(i)}= (\Omega^{{2i-(m+j)-1}}_{O_{X,y}/ \mathbb{Q}})^{\oplus n-1}, for \  \frac{m+j}{2}  < \ i \leq m+j.\\
  \Omega_{O_{X,y}/ \mathbb{Q}}^{\bullet,(i)}= 0, else.
 \end{CD}
\end{cases}
\end{equation} 
\end{corollary}

\subsection{Bloch's formula}
\label{Bloch's formula}
Recall that $X$ is a $d$-dimensional regular scheme of finite type over a field $k$, where $Char k =0$.
Let $T_{j}$ denote the spectrum of the truncated polynomial Spec$(k[t]/(t^{j+1}))$ and  $X_j$ denote the $j$-th infinitesimal neighborhood, i.e. $X_j = X \times T_{j}$. The aim of this subsection is to extend Bloch's formula from $X$ to its infinitesimal neighborhood $X_j$. The tensor triangulated category $D^{perf}(X)$ and $D^{perf}(X_j)$ are equipped with $-codim_{Krull}$ as a dimension function.
\begin{definition}[4, definition 4] or similar definition 3.4.

For any interger $q$, there exists the following augmented Gersten complex $G_{j}$ on the $j$-th infinitesimal neighborhood $X_j$
{\footnotesize
\begin{align*}
 G_{j}: \ & 0 \to  K_{q}(X_{j}) \to K_{q}(k(X)_{j}) \to \bigoplus_{x_{j} \in X_{j} ^{(1)}}K_{q-1}(O_{X_{j},x_{j}} \ on \ x_{j}) \to \dots \\
 & \to \dots \to \bigoplus_{x_{j} \in X_{j} ^{(d)}}K_{q-d}(O_{X_{j},x_{j}} \ on \ x_{j}) \to 0
\end{align*}
}
where $X_{j}=X\times T_{j}$,  $k(X)_{j}= k(X)\times T_{j}$, $x_{j}=x\times T_{j}$.
\end{definition}

The sheafification of this Gersten complex is indeed a flasque resolution as proved below. In the following, we focus on $j=1$ for simplicity. 
\begin{theorem}
 There exists the following splitting commutative diagram(the integer $q \geq 1$):
{\footnotesize
\[
  \begin{CD}
     0 @. 0 @. 0\\
     @VVV @VVV @VVV\\
     \Omega_{X/ \mathbb{Q}}^{\bullet} @<<< K_{q}(X[\varepsilon]) @<<< K_{q}(X) \\
     @VVV @VVV @VVV\\
     \Omega_{k(X)/ \mathbb{Q}}^{\bullet} @<<<  K_{q}(k(X)[\varepsilon]) @<<< K_{q}(k(X)) \\
     @VVV @VVV @VVV\\
     \oplus_{x \in X^{(1)}}H_{x}^{1}(\Omega_{X/\mathbb{Q}}^{\bullet}) @<<< \oplus_{x[\varepsilon]\in X[\varepsilon]^{(1)}}K_{q-1}(O_{X,x}[\varepsilon] \ on \ x[\varepsilon]) @<<<  \oplus_{x \in X^{(1)}}K_{q-1}(O_{X,x} \ on \ x)\\
     @VVV @VVV @VVV\\
      \dots @<<< \dots @<<< \dots \\ 
     @VVV @VVV @VVV\\
     \oplus_{x\in X^{(d)}}H_{x}^{d}(\Omega_{X/ \mathbb{Q}}^{\bullet}) @<<< \oplus_{x[\varepsilon]\in X[\varepsilon]^{(d)}}K_{q-d}(O_{X,x}[\varepsilon] \ on \ x[\varepsilon]) @<<<  \oplus_{x \in X^{(d)}}K_{q-d}(O_{X,x} \ on \ x) \\
     @VVV @VVV @VVV\\
      0 @. 0 @. 0
  \end{CD}
\]
}
where 
\begin{equation}
\begin{cases}
 \begin{CD}
 \Omega_{X/ \mathbb{Q}}^{\bullet} = \Omega^{q-1}_{X/\mathbb{Q}}\oplus \Omega^{q-3}_{X/\mathbb{Q}}\oplus \dots\\
  \Omega_{k(X)/ \mathbb{Q}}^{\bullet} = \Omega^{q-1}_{k(X)/\mathbb{Q}}\oplus \Omega^{q-3}_{k(X)/\mathbb{Q}}\oplus \dots
 \end{CD}
\end{cases}
\end{equation}

\end{theorem}

\begin{proof}
The existence of the right two columns are definition 4.21. The left one, classical Cousin complex, can be obtained by direct computation, using corollary 4.19.
\end{proof}

\begin{theorem}
For each integer $j$, the sheafified Gersten complex $G_{j}$ is a flasque resolution.
\end{theorem}

\begin{proof}
For $j=1$, since $X$ is regular, the sheafifications of both the left and right column in theorem 4.22 are flasque resolutions. So is the middle. For general $j$, the same method works.
\end{proof}

\begin{remark}
The above two theorems answer the following question asked by Green-Griffiths in [16]:
 \begin{quote}
Can one define the Bloch-Quillen-Gersten sequence $G_j$ on 
infinitesimal neighborhoods $X_j = X \times Spec(k[t]/(t^{j+1})$
so that 
\[
 ker(G_1 \to G_0) =  \underline{\underline{T}}G_0,
\]
where $\underline{\underline{T}}G_0$ is a Cousin resolution of differentials.
\end{quote}
The readers can check [16] for more background and [11,36] for discussion 
of the above theorem from different point of view(Chern character, effacement theorem etc).
\end{remark}

Now we can extend Bloch's formula to the infinitesimal neighborhoods $X_{j}$.

\begin{theorem}

\textbf{Bloch's formula}

We have the following identification
\[
  CH_{q}(D^{perf}(X_{j})) = H^{q}(X, K_{q}(O_{X_{j}})).
\]
In particular, for $j=1$,
\[
  CH_{q}(D^{perf}(X[\varepsilon])) = H^{q}(X, K_{q}(O_{X}[\varepsilon])).
\]

\end{theorem}

\begin{proof}
The definition of $CH_{q}(D^{perf}(X_{j}))$ says that it equal to the $q$-th cohomology of the Gersten complex $G_{j}$
\[
 CH_{q}(D^{perf}(X_{j}) = H^{q}(G_{j}).
\]   
It follows because of the fact that the sheafification of $G_{j}$ is a flasque resolution.
\end{proof}

Now, we consider the K-theoretic Chow group as a functor on $X$ and define the tangent space to it as usually.

\begin{definition}
 The tangent space to $CH_{q}(D^{perf}(X))$, denoted by $T_{f}CH_{q}(D^{perf}(X))$, is defined as 
\[
  T_{f}CH_{q}(D^{perf}(X)):= Ker \{Ch_{q}(D^{perf}(X[\varepsilon])) \xrightarrow{\varepsilon =0} CH_{q}(D^{perf}(X)) \}.
\]
\end{definition}

We can identify this tangent space with cohomology group of absolute differentials.
\begin{theorem}
\[
  T_{f}CH_{q}(D^{perf}(X)) = H^{q}(\Omega_{O_{X}/ \mathbb{Q}}^{\bullet}), 
\]
where $\Omega_{O_{X}/ \mathbb{Q}}^{\bullet} = \Omega^{q-1}_{O_{X}/\mathbb{Q}}\oplus \Omega^{q-3}_{O_{X}/\mathbb{Q}}\oplus \dots$
\end{theorem}

\begin{proof}
Diagram chasing. Immediately follows from theorem 4.22 and 4.23.
\end{proof}

Next, we want to use Adams' operation to refiner the diagram in theorem 4.22. Since $q$ can be any integer there, negative K-groups might appear. One can use Weibel's method to extend Adams' operations to negative K-groups, as recalled in section 4.3. In the following, $K^{(i)}_{n}$ denotes the eigen-space of $\psi^{k}=k^{i}$.

\begin{theorem}
There exists the following splitting commutative diagram(the integer $q \geq 1$), each column is a complex whose Zariski sheafification is a flasque resolution. 
{\footnotesize
\[
  \begin{CD}
     0 @. 0 @. 0\\
     @VVV @VVV @VVV\\
     \Omega_{X/ \mathbb{Q}}^{\bullet,(i)} @<<< K^{(i)}_{q}(X[\varepsilon]) @<<< K^{(i)}_{q}(X) \\
     @VVV @VVV @VVV\\
     \Omega_{k(X)/ \mathbb{Q}}^{\bullet,(i)} @<<<  K^{(i)}_{q}(k(X)[\varepsilon]) @<<< K^{(i)}_{q}(k(X)) \\
     @VVV @VVV @VVV\\
     \oplus_{x \in X^{(1)}}H_{x}^{1}(\Omega_{X/\mathbb{Q}}^{\bullet,(i)}) @<<< \oplus_{x[\varepsilon]\in X[\varepsilon]^{(1)}}K^{(i)}_{q-1}(O_{X,x}[\varepsilon] \ on \ x[\varepsilon]) @<<<  \oplus_{x \in X^{(1)}}K^{(i)}_{q-1}(O_{X,x} \ on \ x)\\
     @VVV @VVV @VVV\\
      \dots @<<< \dots @<<< \dots \\ 
     @VVV @VVV @VVV\\
     \oplus_{x\in X^{(d)}}H_{x}^{n}(\Omega_{X/ \mathbb{Q}}^{\bullet,(i)}) @<<< \oplus_{x[\varepsilon]\in X[\varepsilon]^{(d)}}K^{(i)}_{q-d}(O_{X,x}[\varepsilon] \ on \ x[\varepsilon]) @<<<  \oplus_{x \in X^{(d)}}K^{(i)}_{q-d}(O_{X,x} \ on \ x) \\
     @VVV @VVV @VVV\\
      0 @. 0 @. 0
  \end{CD}
\]
}
where 
\begin{equation}
\begin{cases}
 \begin{CD}
 \Omega_{X/ \mathbb{Q}}^{\bullet,(i)}= \Omega^{{2i-q-1}}_{X/ \mathbb{Q}}, for \  \frac{q}{2}  < \ i \leq q.\\
  \Omega_{X/ \mathbb{Q}}^{\bullet,(i)}= 0, else.
 \end{CD}
\end{cases}
\end{equation} 
\end{theorem}

\begin{proof}
We note that negative K-groups(if appear) of the right column theorem 4.22 are 0, so we can use Adam's operations [29], defined at space level, to decompose this column directly. This means that the right column of the above diagram is a complex whose Zariski sheafification is a flasque resolution. 

It's obvious that the left column of the above diagram is a complex whose its Zariski sheafification is a flasque resolution. The differential $\partial_{\varepsilon}$ of the middle column of theorem 4.22 satisfies $\partial_{\varepsilon}=(\delta,\partial)$,
where $\delta$ and $\partial$ are differentials of the left and right columns respectively.
Using corollary 4.19, the middle column is the direct sum of the left and right one. So the middle    column of the above diagram also is a complex whose Zariski sheafification is a flasque resolution.
\end{proof}

In particular, we are interested in  the ``Milnor K-theory''. letting $i=q$, one have the following theorem.

\begin{theorem}
 There exists the following splitting commutative diagram, each column is a complex whose Zariski sheafification is a flasque resolution. Here we assume the integer $q$ satisfies $1 \leq q \leq d$.
{\footnotesize
\[
  \begin{CD}
     0 @. 0 @. 0\\
     @VVV @VVV @VVV\\
     \Omega_{X/ \mathbb{Q}}^{q-1} @<<< K^{M}_{q}(X[\varepsilon]) @<<< K^{M}_{q}(X) \\
     @VVV @VVV @VVV\\
     \Omega_{k(X)/ \mathbb{Q}}^{q-1} @<<<  K^{M}_{q}(k(X)[\varepsilon]) @<<< K^{M}_{q}(k(X)) \\
     @VVV @VVV @VVV\\
     \oplus_{x \in X^{(1)}}H_{x}^{1}(\Omega_{X/\mathbb{Q}}^{q-1}) @<<< \oplus_{x[\varepsilon]\in X[\varepsilon]^{(1)}}K^{M}_{q-1}(O_{X,x}[\varepsilon] \ on \ x[\varepsilon]) @<<<  \oplus_{x \in X^{(1)}}K^{M}_{q-1}(O_{X,x} \ on \ x)\\
     @VVV @VVV @VVV\\
      \dots @<<< \dots @<<< \dots \\ 
     @VVV @VVV @VVV\\
     \oplus_{x \in X^{(q-1)}}H_{x}^{q-1}(\Omega_{X/ \mathbb{Q}}^{q-1}) @<<< \oplus_{x[\varepsilon] \in X[\varepsilon]^{(q-1)}}K^{M}_{1}(O_{X,x}[\varepsilon] \ on \ x[\varepsilon]) @<<< \oplus_{x \in X^{(q-1)}}K^{M}_{1}(O_{X,x} \ on \ x) \\
     @VVV @Vd_{1,q,\varepsilon}^{q-1,-q}VV @Vd_{1,q}^{q-1,-q}VV\\
     \oplus_{x \in X^{(q)}}H_{x}^{q}(\Omega_{X/ \mathbb{Q}}^{q-1}) @<<< \oplus_{x[\varepsilon] \in X[\varepsilon]^{(q)}}K^{M}_{0}(O_{X,x}[\varepsilon] \ on \ x[\varepsilon]) @<<< \oplus_{x \in X^{(q)}}K^{M}_{0}(O_{X,x} \ on \ x) \\
     @VVV @Vd_{1,q,\varepsilon}^{q,-q}VV @Vd_{1,q}^{q,-q}VV\\
     \oplus_{x \in X^{(q+1)}}H_{x}^{q+1}(\Omega_{X/ \mathbb{Q}}^{q-1}) @<<< \oplus_{x[\varepsilon] \in X[\varepsilon]^{(q+1)}}K^{M}_{-1}(O_{X,x}[\varepsilon] \ on \ x[\varepsilon]) @<<< \oplus_{x \in X^{(q+1)}}K^{M}_{-1}(O_{X,x} \ on \ x) \\
     @VVV @VVV @VVV\\
     \dots @<<< \dots @<<< \dots \\ 
     @VVV @VVV @VVV\\
     \oplus_{x\in X^{(d)}}H_{x}^{d}(\Omega_{X/ \mathbb{Q}}^{q-1}) @<<< \oplus_{x[\varepsilon]\in X[\varepsilon]^{(d)}}K^{M}_{q-d}(O_{X,x}[\varepsilon] \ on \ x[\varepsilon]) @<<<  \oplus_{x \in X^{(d)}}K^{M}_{q-d}(O_{X,x} \ on \ x) \\
     @VVV @VVV @VVV\\
      0 @. 0 @. 0
  \end{CD}
\]
}
\end{theorem}

The middle and right columns are complexes, so the definition 4.4 applies. 
\begin{definition}
The Milnor K-theoretic {\it q-cycles} and {\it rational equivalence} of $(X,O_{X})$ are defined to be 
\[
 Z^{M}_{q}(D^{perf}(X)):= Ker(d_{1,q}^{q,-q})
\]
\[
 Z_{q,rat}^{M}(D^{perf}(X)):=Im(d_{1,q}^{q-1,-q}).
\]

The $q^{th}$ Milnor K-theoretic Chow group  of $(X,O_{X})$ is defined to be

\[
 CH^{M}_{q}(D^{perf}(X)):= \dfrac{Z^{M}_{q}(D^{perf}(X))}{Z_{q,rat}^{M}(D^{perf}(X))}.
\]

The Milnor K-theoretic {\it q-cycles} and {\it rational equivalence} of $(X,O_{X}[\varepsilon])$ are defined to be
\[
 Z^{M}_{q}(D^{perf}(X[\varepsilon])):= Ker(d_{1,q,\varepsilon}^{q,-q})
\]
\[
 Z_{q,rat}^{M}(D^{perf}(X[\varepsilon])):=Im(d_{1,q,\varepsilon}^{q-1,-q}).
\]

The $q^{th}$ Milnor K-theoretic Chow group  of $(X,O_{X}[\varepsilon])$ is defined to be

\[
 CH^{M}_{q}(D^{perf}(X[\varepsilon])):= \dfrac{Z^{M}_{q}(D^{perf}(X[\varepsilon]))}{Z_{q,rat}^{M}(D^{perf}(X[\varepsilon]))}.
\]

\end{definition}

\textbf{Agreement}. We now prove that our  Milnor K-theoretic Chow group agrees with the classical ones for regular schemes, after tensoring with $\mathbb{Q}$. 
\begin{theorem} 
For $X$ is a regular scheme of finite type over a field $k$, let $Z^{q}(X)$, $Z_{rat}^{q}(X)$ and $CH^{q}(X)$ denote the classical q-cycles, rational equivalence and Chow groups respectively, after tensoring with $\mathbb{Q}$, then we have the following identifications
\[
 Z_{q}^{M}(D^{perf}(X))= Z^{q}(X)
\]

\[
  Z_{q,rat}^{M}(D^{perf}(X)) = Z_{rat}^{q}(X)
\]

\[
 CH^{M}_{q}(D^{perf}(X)) = CH^{q}(X).
\]
\end{theorem}

\begin{proof}
Since $X$ is a regular, Soul\'e's Riemann-Roch without denominator [29, or theorem 4.3] shows that the right column of theorem 4.29 agrees with the following classical sequence, ignoring torsion

\begin{footnotesize}
\begin{align*}
 0 \to &K^{M}_{q}(O_{X}) \to K^{M}_{q}(k(X)) \to \bigoplus_{x \in X^{(1)}}K^{M}_{q-1}(k(x)) \to \dots   \\
 &\to \bigoplus_{x \in X^{(q-1)}}K^{M}_{1}(k(x)) \xrightarrow{d_{1,q}^{q-1,-q}} \bigoplus_{x \in X^{(q)}}K^{M}_{0}(k(x)) \xrightarrow{d_{1,q}^{q,-q}} 0.
\end{align*}
\end{footnotesize}

Noting $K^{M}_{0}(k(x))= K_{0}(k(x))=\mathbb{Z}$, one has
\[
  Z^{M}_{q}(D^{perf}(X))= Ker(d_{1}^{q,-q})= \bigoplus_{x \in X^{(q)}}K^{M}_{0}(k(x)) = Z^{q}(X).
\]

Since $K^{M}_{1}(k(x))= K_{1}(k(x))$, as explained in Quillen's proof of Bloch's formula [25], the image of $d_{1}^{q-1,-q}$ gives the rational equivalence. Hence, 
\[
 Z^{M}_{q,rat}(D^{perf}(X)) = Im(d_{1}^{q-1,-q}) = Z^{q}_{rat}(X).
\]

Therefore, we have the following identification
\[
  CH^{M}_{q}(D^{perf}(X)) = CH^{q}(X).
\]
\end{proof}

Also we obtain Bloch's formulas. This gives a positive answer to Green-Griffiths' \textbf{Question 1.1} on page 2.

\begin{theorem}

\textbf{Bloch's formula}

After tensoring with $\mathbb{Q}$, we have the following identifications:
\[ 
 CH^{M}_{q}(D^{perf}(X)) = H^{q}(X,K_{q}^{(q)}(O_{X})) = H^{q}(X,K_{q}^{M}(O_{X})).
\]

 \[
  CH^{M}_{q}(D^{perf}(X[\varepsilon]))=H^{q}(X,K_{q}^{(q)}(O_{X}[\varepsilon]))= H^{q}(X,K_{q}^{M}(O_{X}[\varepsilon])) .
 \]
\end{theorem}

\begin{proof}
Immediately from the flasque resolutions of sheafifications of columns in theorem 4.29.
\end{proof}

Now, we define the tangent space to Milnor K-theoretic Chow group.

\begin{definition}
 The tangent space to $CH^{M}_{q}(D^{perf}(X))$, denoted by $T_{f}CH^{M}_{q}(D^{perf}(X))$, is defined to be
\[
 T_{f}CH^{M}_{q}(D^{perf}(X)) := Ker \{CH^{M}_{q}(D^{perf}(X[\varepsilon])) \xrightarrow{\varepsilon =0} CH^{M}_{q}(D^{perf}(X)) \}.
\]
\end{definition}

Recall that one can formally  define tangent space to the classical Chow group as 
\[
 T_{f}CH^{q}(X)= H^{q}(X,TK_{q}^{M}(O_{X}))= H^{q}(X, \Omega_{O_{X}/ \mathbb{Q}}^{q-1}).
\]

Our definition agrees with it: 
\begin{theorem}
\[
 T_{f}CH^{M}_{q}(D^{perf}(X)) = H^{q}(X, \Omega_{O_{X}/ \mathbb{Q}}^{q-1}) = T_{f}CH^{q}(X).
\]
\end{theorem}

\begin{proof}
Diagram chasing. Immediately from theorem 4.29.
\end{proof}

We continue studying the geometry of these Milnor K-theoretic Chow groups and their tangent spaces in forthcoming papers.

\end{document}